\documentclass[a4paper,12pt,oneside,dvipsnames,reqno]{amsart} 

\usepackage{amsmath,amsfonts,amsthm,amssymb}

\usepackage{xcolor}
\usepackage{graphicx}
\usepackage{cite}

\usepackage{comment}

\usepackage{bbm}
\usepackage{tikz-cd}

\usepackage{longtable}

\usepackage{hyperref}
\hypersetup{colorlinks=true, citecolor=PineGreen, linkcolor=RoyalBlue, linktoc=page,urlcolor=RoyalBlue}

\usepackage{geometry}

\theoremstyle{plain}
\newtheorem{theorem}{Theorem}[section] 
\newtheorem{lemmy}[theorem]{Lemma}

\newtheorem{cor}[theorem]{Corollary}

\theoremstyle{remark}
\newtheorem{rem}[theorem]{Remark}

\theoremstyle{definition}

\newcommand{\A}{\mathbb{A}}

\newcommand{\R}{\mathbb{R}}

\newcommand{\C}{\mathbb{C}}

\newcommand{\Z}{\mathbb{Z}}

\newcommand{\Q}{\mathbb{Q}}

\newcommand{\GL}{\operatorname{GL}}

\newcommand{\SL}{\operatorname{SL}}

\newcommand{\Mat}{\operatorname{Mat}}

\newcommand{\N}{\mathbb{N}}

\newcommand{\Ind}{\operatorname{Ind}}

\newcommand{\sgn}{\operatorname{sgn}}

\newcommand{\Vol}{\operatorname{Vol}}

\newcommand{\abs}[1]{\lvert{#1}\rvert}
\newcommand{\abss}[1]{\left\lvert{#1}\right\rvert}

\title{The sup-norm problem beyond the newform}

\author{Edgar Assing}

\newcommand{\Date}{21$^{\mathrm{st}}$ February 2022}
\date{\Date}

\begin{document}

\begin{abstract}
In this note we take up the classical sup-norm problem for automorphic forms and view it from a new angle. Given a twist minimal automorphic representation $\pi$ we consider a special small $GL_2(\Z_p)$-type $V$ in $\pi$ and prove global sup-norm bounds for an average over an orthonormal basis of $V$. We achieve a non-trivial saving when the dimension of $V$ grows.
\end{abstract}
	
\maketitle

\section{Introduction}

It is a classical problem in analysis and mathematical physics, more precisely Quantum Chaos, to bound the $L^{\infty}$-norm of certain eigenfunctions on manifolds. In the most basic situation one considers a Riemann surfaces $X$ of finite volume and eigenfunctions $\phi$ of the Laplace-Beltrami operator $\Delta_X$. A sup-norm bound in the spectral aspect is then an estimates of the form
\begin{equation}
	\frac{\Vert \phi \Vert_{\infty}}{\Vert \phi \Vert_2} \ll_X (1+\abs{t_{\phi}})^{\frac{1}{2}-\delta+\epsilon}, \label{eq:sup-norm_prob}
\end{equation}  
where $\lambda_{\phi}=\frac{1}{4}+t_{\phi}^2$ is the Laplace-Beltrami eigenvalue of $\phi$. The local bound corresponds to $\delta=0$ and is known in great generality. The sup-norm problem asks for improved bounds featuring some $\delta>0$. The sup-norm problem has only been solved for very special surfaces $X$ and is hopeless in general. Indeed there is a well known obstruction to the sup-norm problem coming from large eigenspaces $V_{\lambda}$ given by the inequality
\begin{equation}
	\dim_{\C} V_{\lambda} \ll_X \sup_{\phi\in V_{\lambda}} 	\frac{\Vert \phi \Vert_{\infty}}{\Vert \phi \Vert_2}. \nonumber
\end{equation}
This observation is enough to establish the well known fact, that the local bound (i.e. \eqref{eq:sup-norm_prob} with $\delta=0$) can not be improved for the sphere $X=S^2$. So far we have only described the most basic version of the sup-norm problem which, is already very interesting on its own. In addition it admits many variations which have been studied throughout the years. An example for such a variation is the so called level aspect where the base manifold changes in some convenient family $X_1,X_2,\ldots$ and one keeps track of this change in the sup-norm bound \eqref{eq:sup-norm_prob} using a suitable parameter called the level. Another generalisation that should be mentioned allows $X$ to be a manifold of higher dimension and rank.  

Essentially any progress that has been made towards the sup-norm problem as introduced above relies on the arithmeticity of $X$. The basic idea introduced in the monumental paper \cite{iwaniec-sarnak} is to employ additional symmetries (in the  form of Hecke operators) to build a spectral projector that is sharper than the one constructed with only the Laplace-Beltrami operator at hand. Morally this might be thought of as forcing a multiplicity one situation even if the Laplace-Beltrami eigenspaces can not be rigorously controlled. The result of this method is a bound as in \eqref{eq:sup-norm_prob} with $\delta=\frac{1}{12}$ for compact quotients $X=\Gamma\backslash \mathbb{H}$ constructed from maximal orders in quaternion algebras.

Since its appearance the method from \cite{iwaniec-sarnak} has been tweaked, modified and generalised, see for example \cite{blomer-holowinsky,templier_hybrid,saha, assing_sup1, blomer-maga} and the references within. Much work is concerned with congruence quotients $X=\Gamma_0(N)\backslash \mathbb{H}$ on which so called Hecke-Maa\ss\  newforms are considered. Since these newforms enjoy a nice multiplicity one property they are natural candidates for the sup-norm problem. In this note we are going beyond the case of newforms and consider situations where the dimension of the underlying $p$-adic representation grows. In other words, we solve the sup-norm problem in the dimension aspect. This aspect is a new facet of the sup-norm problem which seems extremely interesting and is not yet well studied. While our result is the first in $p$-adic setting it is only preceded by \cite{blomer-maga-harcos-milicevic} where an archimedean version of this aspect is discussed.  

To explain our result and its connection to the work of Blomer, Harcos, Maga and Mili\'cevi\'c it will be most convenient to leave the classical world of Hecke-Maa\ss-newforms behind and work in the language of automorphic forms and automorphic representations.

The sup-norm problem we will consider is connected to small $\GL_2(\Z_p)$-types in cuspidal automorphic representations $\pi$, where $p>3$ is prime. Comparing this to the recent work \cite{blomer-maga-harcos-milicevic} we are replacing the archimedean place $\infty$ by a finite place $p$ and the minimal $U(2)$-type of some automorphic representation by a suitably chosen $\GL_2(\Z_p)$-type. Note that in order to afford interesting $K$-types at the archimedean place it is necessary to work over fields admitting complex places or in higher rank. In the $p$-adic world we already meet interesting cases when working with automorphic forms for $\GL_2$ over $\Q$.

\subsection{Set-up and main result}

Before we continue our discussion we need to fix some notation. Let $G(R)=\GL_2(R)$ for some ring $R$ and let $\A$ be the adele ring over $\Q$. We will be working with cuspidal automorphic representations $\pi$ of $G(\A)$ with unitary central character $\omega_{\pi}$. Abusing notation we will write $\pi \subset L_0^2(G(\Q)\backslash G(\A),\omega_{\pi})$ assuming that $\pi$ acts on an irreducible subspace of cuspidal automorphic forms by right translation. Given a compact subgroup $H$ we write $\pi^H$ for the space of $H$-invariant elements in $\pi$.

Set $K_{\infty}=SO(2)$ and $K_l=\GL_2(\Z_l)$ for primes $l$. Combining these we get the compact subgroup $K=\prod_v K_v\subset G(\A)$. Given a prime $p>3$ and $m>0$ we consider the smaller compact subgroup $$K(p^m)=K_{\infty}\times K_p(m) \times \prod_{l\neq p} K_l \text{ for }K_p(p^m) =1+p^m\cdot \Mat_{2\times 2}(\Z_p)\subset K_p.$$ Note that $K(p^m)$ is normal and of finite index in $K$.

Throughout we restrict ourselves to the situation where $\pi$ is unramified (i.e. spherical) away from $p$. In particular it is spherical at $\infty$ and one associates the spectral parameter $t_{\pi}$. Set $T=1+\abs{t_{\pi}}$. Further, we have
\begin{equation}
	m_{\pi} = \min \{m\in\N\colon \pi^{K(p^m)}\neq \{0\} \}< \infty.\nonumber
\end{equation}
We set $V=\pi^{K(p^{m_{\pi}})}$ and observe that $\pi\vert_K$ endows $V$ with the structure of a $K$-module. It turns out that, if $\pi$ is twist minimal, $V$ is irreducible (see Lemma~\ref{lm:irred} below). Set $d=\dim_{\C}V$, choose an orthonormal basis $\phi_1,\ldots,\phi_d$ for $V$ with respect to the $L^2_0(G(\Q)\backslash G(\A),\omega_{\pi})$ inner product. Define 
\begin{equation}
	\Phi(g) = \left( \sum_{i=1}^d \abs{\phi_i(g)}^2\right)^{\frac{1}{2}}.\nonumber
\end{equation}
We are concerned with the sup-norm of $\Phi(g)$ and obtain the following theorem which is a close analogue to \cite[Theorem~1]{blomer-maga-harcos-milicevic}.

\begin{theorem}\label{th:main}
Let $p>3$ be prime and suppose $\pi$ is twist minimal. In the notation above we have
\begin{equation}
	\Vert \Phi \Vert_{\infty} \ll T^{\frac{1}{2}+\epsilon}d^{\frac{11}{12}+\epsilon}.\nonumber
\end{equation}
If the (arithmetic)-conductor of $\pi$ is a perfect square (i.e. the exponent-conductor of the $p$-component $\pi_p$ of $\pi$ is even) or the $p$-component $\pi_p$ of $\pi$ is not supercuspidal, then we have the better bound
\begin{equation}
	\Vert \Phi \Vert_{\infty} \ll T^{\frac{1}{2}+\epsilon}d^{\frac{5}{6}+\epsilon}.\label{stronger}
\end{equation}
\end{theorem}

While in the spectral aspect (i.e. the $T$-aspect in our statement) we only recover the local bound, the key feature of our theorem is the sub-local exponent in the dimension aspect $d$. Given the obstruction to the sup-norm problem coming from growing eigenspaces the aspect under consideration may seem counter intuitive. However, we are letting the dimension of the eigenspace vary in a controlled manner and manage to show that one can still achieve a considerable power saving in $d$ on average over any orthonormal basis.

Note that the sup-norm bound given in the theorem holds globally. Thus, unlike the one in \cite{blomer-maga-harcos-milicevic}[Theorem~1], no restriction to a compact domain is necessary here.  As usual when proving global sup-norm bound the argument consists of two steps. First, a bound via the Whittaker expansion takes care of the regions close to the cusps. This part of the argument is fairly standard but requires some new computations of ramified Whittaker vectors. Second, a bound obtained from the amplified pre-trace inequality is used to handle the bulk. At this point it becomes crucial that we are only treating the average function $\Phi$. Indeed, this allows us to identify the test function on the geometric side as a character of a finite group. The analysis of this character is carried out in Lemma~\ref{character} below and relies on character tables given in \cite{kutzko}. This is the only place where the assumption $p>3$ is used.

To end this section let us briefly discuss the numerology of the exponents in the $d$-aspect. For simplicity we restrict this discussion to the cases in which our result gives the strong bound \eqref{stronger}. Let us start by talking about the local- (not to say trivial-) bound (in the bulk). To obtain this we can follow Marshall's strategy (see \cite{marshall}) which leads to the following. Let $F$ be any cuspidal automorphic form so that the translates $\phi(\cdot k)$, $k\in K$, generate an irreducible $K$-module $W_F$. Then choosing certain $K$-matrix-coefficients as test functions in the pre-trace inequality yields
\begin{equation}
	\frac{\Vert F\vert_{\Omega} \Vert_{\infty}}{\Vert F\Vert_2} \ll \dim_{\C}( W_F)^{\frac{1}{2}}.\label{eq:local}
\end{equation}
Applying this to $\Phi$ upon noting that $\Vert \Phi \Vert_2= d^{\frac{1}{2}}$ suggests the local bound
\begin{equation}
	\Vert \Phi \Vert_{\infty} \ll d^{1+\epsilon}.\nonumber
\end{equation}
(The same bound can also be obtained from the Whittaker expansion coupled with a suitable generating domain.) Thus amplification allows us to improve the exponent from the local bound by $\frac{1}{6}$, which should be an familiar exponent. More suggestively we can write our main result as $$\frac{\Vert \Phi\Vert_{\infty}}{ \Vert \Phi \Vert_2} \ll d^{\frac{1}{2}-\frac{1}{6}+\epsilon}.$$

One could say that Theorem~\ref{th:main} implies $\Vert \phi_i\Vert_{\infty}\ll d^{\frac{1}{3}}$ on average. Note that if $p^{m_{\pi}}$ agrees with the arithmetic conductor $p^{n_{\pi}}$ of $\pi$, then this result is not very interesting. Indeed, in this case we can generate the elements $\phi_1,\ldots,\phi_d$ in $V$ directly from the newform $\phi_{\circ}$ in $\pi$. By now there are very good bounds for this newform (and thus also for the $\phi_i$'s) known in the literature. See \cite{templier_hybrid} if $m_{\pi}=n_{\pi}=1$ or \cite{comtat} in general. However, in the remaining cases (since $\pi$ is assumed to be twist minimal these correspond to the situation where $\pi$ is supercuspidal at $p$) our result provides new information in the sup-norm problem. Indeed one can still generate $V$ from a translate of the newfom $\phi_{\circ}$. (This is precisely the strategy used in \cite{marshall,saha} to derive local bounds for the newform of arbitrary level using \eqref{eq:local}.) Translated into the level-aspect our result now essentially says that the sup-norm of the $\phi_i$'s is bounded by $p^{\frac{1}{3}\lceil \frac{n_{\pi}}{2}\rceil}$ on average. To the best of our knowledge this can not be derived from any known sup-norm results on the newform $\phi_{\circ}$.

Finally we want to compare our result to the guiding archimedean example \cite[Theorem~1]{blomer-maga-harcos-milicevic}. Recall that we need to replace the $K$-module $V$ by some irreducible $U(2)$ representation $W$. This representation $W$ will occur as the minimal $U(2)$-type in some cuspidal automorphic $\pi$ of $G(\A_{\Q(i)})$. Note that if $\dim_{\C}W\asymp l$ we can think of $\pi$ (or rather $\pi_{\infty}$) having spectral density $\asymp l^2$. This explains the local bounds
\begin{equation}
	\frac{\Vert \Phi \Vert_{\infty}}{\Vert \Phi \Vert_{2}} \ll l^{1+\epsilon} \text{ or } \Vert \Phi\Vert_{\infty} \ll l^{\frac{3}{2}+\epsilon}, \nonumber
\end{equation}
where $\Phi$ is constructed as an average over some suitable basis of $W$ similar to our construction above. As result of an amplification process the authors of \cite{blomer-maga-harcos-milicevic} arrive at
\begin{equation}
	\frac{\Vert \Phi\vert_{\Omega}\Vert_{\infty}}{\Vert \Phi\Vert_2} \ll (l^2)^{\frac{1}{2}-\frac{1}{12}+\epsilon}.\nonumber
\end{equation}
Our notation suggests that in the result from \cite{blomer-maga-harcos-milicevic} the number $l^2$ playes the role of our $d$. This can be explained via the spectral density of $\pi_{\infty}$ and respectively $\pi_p$. Indeed while in the archimedean situation the spectral density is roughly $l^2$ in our case the spectral density is linearly related to $d$. Thus in both cases the square root of the spectral density seems to determine the trivial bound. (This is only reasonable because we are considering minimal or close to minimal $K$-types in both cases.) Note that the quality of the saving $\frac{1}{6}$ in the $p$-adic versus $\frac{1}{12}$ in the archimedean case comes from slightly different behaviour of the spectral transform. 

Finally, let us remark that if the exponent conductor of $\pi_p$ is odd and $\pi_p$ is supercuspidal, then our bounds for the spectral transform, which in this case are linked to certain badly-behaved characters of $\GL_2$ over finite rings, are comparable to those used in \cite{blomer-maga-harcos-milicevic}. This explains that in this case we have matching numerology and obtain only a saving of $\frac{1}{12}$ in the final exponent. Translated to the level aspect our result states that on average the $\phi_i$'s are bounded by $p^{\frac{5(n_{\pi}+1)}{24}}$. Bounds of this quality are known for newforms only in the compact setting, see \cite{saha-hu}.

\begin{rem}
Questions of these type should be even more interesting when considered in higher rank. The reason is that in higher rank the analogously defined small $K$-types can not be generated from translates of the newform. For example if one considers a depth-zero supercuspidal representation $\pi_p$ of $\GL_3(\Q_p)$, then it has (arithmetic)-conductor $p^3$ and the space $\pi_p^{K_p^{(3)}(1)}$, where $K_p^{(3)}(1)$ is the principal congruence subgroup modulo $p$ in $\GL_3(\Z_p)$, is non-zero. However, it seems impossible to find a translate of the newform that generates $\pi_p^{K_p^{(3)}(1)}$. Indeed this would mean finding $g\in \GL_3(\Q_p)$ with $$K_p^{(3)}(1)\subset g^{-1}\left[ \begin{matrix} \Z_p & \Z_p & \Z_p \\ \Z_p & \Z_p & \Z_p \\ p\Z_p & p\Z_p & 1+p\Z_p\end{matrix}\right] g.$$ However, the question treated in this paper still makes sense and trying to answer it is work in progress.
\end{rem}

\noindent\textbf{Acknowledgments:} We would like to thank Prof. Dr. V. Blomer for fruitful discussions and useful comments on an earlier draft of this manuscript. I would also like to thank the anonymous referee for pointing out an oversight in the amplification argument which has now been fixed.

\section{Preliminary considerations}

In this section we are putting in some ground work on which the following sections will rely. 

Recall that $\pi$ was a cuspidal automorphic representation. Since we are assuming that $\pi_v$ is unramified for $v\neq p$, the (arithmetic)-conductor of $\pi$ is $p^{n_{\pi}}$ for $n_{\pi}\in \N\cup\{0\}$. When $n_{\pi}=0$ we have $d=1$ and our theorem reduces to the local bound in the spectral aspect, so that without loss of generality we can assume $n_{\pi}\geq 1$ throughout. By Flath's factorisation theorem we can fix an isomorphism $\pi \cong \bigotimes \pi_v$. Note that also the central character of $\pi$ factors as $\omega_{\pi}=\bigotimes_v \omega_{\pi_v}$ where $\omega_{\pi_v}$ is the central character of $\pi_v$.  For $v\neq p$ we can fix a spherical (i.e. $K_v$-invariant vector)  $\phi_v^{\circ}\in \pi_v^{K_v}$. This vector is unique up to scaling. Recall that $\phi_i$, $i=1,\ldots,d$ forms an orthonormal basis of $V=\pi^{K(p^{m_{\pi}})}$. Thus there is $\phi_p^{(i)}$ so that we can identify
\begin{equation}
	\phi_i = \phi_p^{(i)} \cdot \prod_{v\neq p} \phi_v^{\circ}.\nonumber
\end{equation}
Since the spherical functions $\phi_p^{\circ}$ are well understood much of our work boils down to understanding properties of an orthonormal basis
\begin{equation}
	\text{span}\{ \phi_p^{(1)},\ldots,\phi_p^{(d)} \} = \pi_p^{K_p(m_{\pi})}.\nonumber
\end{equation}
This a purely local problem, which we investigate in the following subsection.

\subsection{Local considerations}

We now focus on properties of the local representation $\pi_p$. We start by recalling the classification of local representations. But before we do so we need some more notation. Given a (quasi)-character $\chi\colon \Q_p^{\times}\to \C^{\times}$ we write $a(\chi)$ for the (exponent)-conductor. Further write $$I_0(p)=\left[\begin{matrix} \Z_p^{\times} &\Z_p \\ p\Z_p & \Z_p^{\times}\end{matrix}\right]\subset K_p$$ for an Iwahori subgroup and put $K_p'=N_{G(\Q_p)}(I_0(p))$. We also need the filtration
\begin{equation}
	K_p'(m) = 1+\left[\begin{matrix}p\Z_p & \Z_p \\ p\Z_p & p\Z_p\end{matrix}\right]^m\nonumber
\end{equation}
of $K_p'$ by normal subgroups. Finally given two quasi characters $\chi_1,\chi_2\colon \Q_p^{\times}\to\C^{\times}$ we form the (normalised) induced representation on $\Ind_B^{G(\Q_p)}(\chi_1\otimes \chi_2)$ as usual. If this representation is irreducible, then we denote the so obtained representation by $\chi_1\boxplus\chi_2$. We write $\text{St}$ for the Steinberg representation which we may identify with the unique irreducible subspace of $\Ind_B^{G(\Q_p)}(\abs{\cdot}^{\frac{1}{2}}\otimes \abs{\cdot}^{-\frac{1}{2}})$. We are now ready to recite the following well known classification.

\begin{lemmy}
The representation $\pi_p$ falls into one of the following three cases:
\begin{itemize}
	\item \textbf{Case 1 (Principal Series):} There are (quasi)-characters $\chi_i\colon \Q_p^{\times}\to \C^{\times}$ such that $\chi_1\chi_2=\omega_{\pi_p}$, $a(\chi_1)+a(\chi_2)=n_{\pi}$ and  $\pi_p = \chi_1\boxplus \chi_2$.
	\item \textbf{Case 2 (Special):} There is a (quasi)-character $\chi\colon \Q_p^{\times}\to \C^{\times}$ with $n_{\pi}=2a(\chi)$ if $a(\chi)>0$ or $n_{\pi}=1$ otherwise, $\chi^2=\omega_{\pi_p}$ and $\pi_p=\chi\otimes \text{St}$.
	\item \textbf{Case 3 (Supercuspidal):} The representation $\pi_p$ is supercuspidal. In this case we can write $\pi_p=\chi \cdot \pi_p'$ for a (quasi)-character $\chi\colon \Q_p^{\times}\to\C^{\times}$ and some twist-minimal representation $\pi_p'$ of conductor $n_{\pi}'$ which is constructed in one of the following two ways:
	\begin{itemize}
		\item \textbf{Case 3.1 ($n_{\pi}'$ even):} There is an irreducible representation $\tau$ of $Z\cdot K_p$ with $\tau\vert_{Z} = \chi^{-2}\cdot \omega_{\pi_p}$ which is invariant by $K_p(\frac{n_{\pi}'}{2})$ so that $\pi_p' = c-\Ind_{ZK_p}^{G(\Q_p)}\tau$.
		\item \textbf{Case 3.2 ($n_{\pi}'$ odd):} There is an irreducible representation $\tau$ of $K_p'$ which is invariant by $K_p'(n_{\pi}'-1)$ with $\tau\vert_Z=\chi^{-2}\cdot \omega_{\pi_p}$ such that $\pi_p'=c-\Ind_{K_p'}^{G(\Q_p)}\tau$.
	\end{itemize}
\end{itemize}
\end{lemmy}

With this classification at hand we continue to study the subspaces $V$ in more detail.

\begin{lemmy}\label{lm:irred}
Suppose $\pi_p$ is twist minimal, then $V=\pi_p^{K(m_{\pi})}$ is irreducible as $K_p$-module and we have:
\begin{enumerate}
	\item The invariant $m_{\pi}$ is given by $$m_{\pi}=\begin{cases} n_{\pi} &\text{ if } \pi_p \text{ is Case~1},\\ \lfloor \frac{n_{\pi}+1}{2}\rfloor & \text{ if }\pi_p \text{ is Case~2,3.}\end{cases}$$
	\item The dimension of $V$ is given by $$d= p^{m_{\pi}}\cdot \begin{cases} (1+\frac{1}{p}) &\text{ if $\pi_p$ is Case~1,}\\1 &\text{ if $\pi_p$ is Case~2,} \\ (1-\frac{1}{p}) &\text{ if $\pi_p$ is Case~3.1 and }\\ (1-\frac{1}{p^2}) &\text{ if $\pi_p$ is Case~3.2.} \end{cases}$$
\end{enumerate}
\end{lemmy}
\begin{proof}
This is not new and we only have to ensemble the pieces appropriately. Let us proceed case by case.

First, if $\pi_p$ is in Case~1, then twist-minimality implies that $\chi_2$ (or similarly $\chi_1$) is unramified. Thus we have $n_{\pi}=a(\chi_1)$ and the results on $d$ and $m_{\pi}$ follow from \cite[Proposition~4.3]{miyauchi-yamauchi_remark}. Irreducibility can be seen by direct computation.

Second, if $\pi_p$ is in Case~2 and twist minimal, then $\pi_p = \text{St}$ and $n_{\pi}=1$. The results on $d$ and $m_{\pi}$ follow again from \cite[Proposition~4.3]{miyauchi-yamauchi_remark}. In this case irreducibility follows from \cite[Theorem~1]{casselman}.

Finally, if $\pi_p$ belongs to Case~3, then the full statement is given in \cite[Theorem~3.5]{loeffler_weinstein}. (See also \cite[Lemma~4.5, Corollary~4.7]{miyauchi-yamauchi_remark} for the computation of $m_{\pi}$ and $d$.)
\end{proof}

\subsection{A generating domain}

We now switch to the global picture again and aim to produce a suitable set $\mathcal{F}\subset G(\A)$ which reduces our problem to studying 
\begin{equation}
	\mathcal{S}(\Phi,\mathcal{F}) = \sup_{g\in\mathcal{F}} \abs{\Phi(g)}.\nonumber
\end{equation}

Let $\mathcal{F}$ be the standard fundamental domain for $\SL_2(\Z)\backslash \mathbb{H}$, which we identify with a subset of $\GL_2(\R)$ by identifying $z=x+iy\in\mathbb{H}$ with $n(x)a(y)\in B(\R)\subset\GL_2(\R)$. Here 
\begin{equation}
	n(x) =\left(\begin{matrix}
	 1&x\\0&1\end{matrix}\right) \text{ and } a(y) =\left(\begin{matrix}
	 y&0\\0&1\end{matrix}\right).\nonumber
\end{equation}
We further view $\mathcal{F}$ as a subset of $G(\A)$ by identifying it with its image under the usual embedding $G(\R) \to G(\A)$. The same series of identifications allows us to write $\Phi(z)$ for $z\in \mathbb{H}$.

\begin{lemmy}
Suppose $$\mathcal{S}(\Phi,\mathcal{F})\leq A$$  holds for $\Phi^2=\sum_{i=1}^d\abs{\phi_i}^2$ constructed from an arbitrary orthonormal basis $\phi_1,\ldots,\phi_d$ of $V$, then  
\begin{equation}
	\Vert \Phi \Vert_{\infty} \leq A.\nonumber
\end{equation}
\end{lemmy}
\begin{proof}
First we take $g\in G(\A)$ and observe that by strong approximation we can write 
\begin{equation}
	g=\gamma zb k \text{ with }\gamma\in G(\Q),\, z\in Z(\R),\, b\in\mathcal{F} \text{ and }k\in K.\nonumber
\end{equation}
We directly obtain $\abs{\phi_i(g)}=\abs{\phi_i(bk)}$ by automorphy and the action of $Z$ via a unitary character. However, we now observe that if $\phi_1,\ldots,\phi_d$ forms an orthonormal basis of $V$, then so does $\pi(k)\phi_1,\ldots, \pi(k)\phi_d$. Let us write $\Phi^{(k)}$ for the average constructed from the latter basis. We thus have
\begin{equation}
	\Phi(g) \leq \mathcal{S}(\Phi^{(k)}, \mathcal{F})\leq A\nonumber
\end{equation}
by assumption.
\end{proof}

\section{The Whittaker bound}

We will now start the process of deriving a first bound for $\Phi$ which will be valid (high) up in the cusp. This is done by estimating $\Phi$ using the Whittaker expansions of the $\phi_i$'s. Throughout we will be working with an arbitrary orthogonal basis $\phi_1,\ldots,\phi_d$ and consider only $g\in \mathcal{F}$.

\subsection{Reduction to a local problem}

Let $\phi=\phi_i$ for some $i=1,\ldots, d$. The global Whittaker period is given by
\begin{equation}
	W_{\phi}(g) =  \int_{\Q\backslash \A} \phi(n(x)g)\psi_{\A}(x)^{-1}dx.\nonumber
\end{equation}
Where $\psi_{\A}$ is the standard character of $\Q\backslash\A$ which has a factorisation $\psi_{\A} = \bigotimes_{v}\psi_v$ for $\psi_{\infty}(x_{\infty}) = e(x_{\infty})$ and $\psi_l$ unramified for all primes $l$. Note that $W_{\phi}(\cdot)$ is right $K(p^{m_{\pi}})$-invariant and transforms with respect to $\psi_{\A}$ when acted on by $N(\A)$ from the left. Thus a standard trick shows that $W_{\phi}(a(q)g_{\infty})=0$ unless $0\neq q\in \frac{1}{p^{m_{\pi}}}\Z$. Indeed, for any $x$ with $n(x)\in K(p^{m_{\pi}})$, one computes
\begin{equation}
	W(a(q)g_{\infty})=W(a(q)g_{\infty}n(x)) = W(n(xq)a(q)g_{\infty}) = \psi_{\A}(xq)W(a(q)g_{\infty}).\nonumber
\end{equation}
We conclude that, if $W(a(q)g_{\infty})\neq 0$, then we have $\psi_{\A}(xq)=1$ for all such $x$. This gives precisely the condition $q\in \frac{1}{p^{m_{\pi}}}\Z$.

This observation leads to the Whittaker expansion
\begin{equation}
	\phi(g_{\infty}) = \sum_{n\in \Z\setminus \{0\}} W_{\phi}\left(a(\frac{n}{p^{m_{\pi}}})g_{\infty}\right).\nonumber
\end{equation}

We need to exploit the factorisation of the Whittaker function $W_{\phi}$. To do so we first observe that we have the factorisation of Whittaker models
\begin{equation}
	\mathcal{W}(\pi,\psi_{\A}) = \bigotimes_v \mathcal{W}(\pi_v,\psi_v).\nonumber
\end{equation}
Using the factorisation of $\phi$ will now determine distinguished elements in the local Whittaker models as follows. Starting at $v=\infty$ we set
\begin{equation}
	W_v(n(x)a(y)) = \frac{\abs{y}^{\frac{1}{2}}K_{it_{\pi}}(2\pi\abs{y})}{2\abs{\Gamma(\frac{1}{2}+it_{\pi})\Gamma(\frac{1}{2}-it_{\pi})}^{\frac{1}{2}}}e(x),\nonumber
\end{equation}
where $t_{\pi}$ is the spectral parameter of $\pi_{\infty}$. Of course $W_{v}$ is the spherical Whittaker function and is normalised so that
\begin{equation}
	\int_{\R^{\times}} \abs{W_{v}(a(y))}^2\frac{dy}{\abs{y}} = 1, \nonumber
\end{equation}
where $dy$ is the normal Lebesgue measure.

We turn towards the finite places $v\neq p$ given by some prime $l\neq p$. The spherical Whittaker function in $\mathcal{W}(\pi_v,\psi_v)$ is then given by
\begin{equation}
	W_v(a(y)) = \abs{y}_v^{\frac{1}{2}}\lambda_{\pi}(l^{v_l(y)}).\nonumber
\end{equation}
Here $\lambda_{\pi}(n)$ is defined by $$L^{\{\infty,p\}}(s,\pi)= \prod_{v\neq p,\infty} L_v(s,\pi_v)=\sum_{(n,p)=1}\lambda_{\pi}(n)n^{-s}$$ in analytic normalisation. We have set things up so that $W_v(1)=1$.

Finally we turn towards $v=p$. Here we write $W_p^{(i)}$ for the image of $\C \phi_p^{(i)}$ in the Whittaker model $\mathcal{W}(\pi_p,\psi_p)$ such that
\begin{equation}
	\langle W_p^{(i)},W_p^{(i)}\rangle_{\mathcal{W}(\pi_p,\psi_p)} = \int_{\Q_p^{\times}}\abs{W_p^{(i)}(a(y))}^2 \frac{dy}{\abs{y}}=1,\nonumber
\end{equation}
here $dy$ is the Haar measure of $\Q_p$ normalised so that $\Vol(\Z_p,dy)=1$.

With these choices made there is are constants $C_{\pi}^{(i)}\in \C^{\times}$ so that
\begin{equation}
	\frac{W_{\phi_i}(g)}{\Vert \phi_i \Vert_2} = C_{\pi}^{(i)} \cdot \prod_{v}W_v(g_v), \nonumber
\end{equation}
As shown in \cite[Section~4]{lapid-mao_fourier} the absolute values of these constants satisfy
\begin{equation}
	\abs{C_{\pi}^{(i)}}^{2} = \lim_{s\to 1} \frac{\zeta^{\{p,\infty\}}(1)\zeta^{\{p,\infty\}}(2)}{L^{\{p,\infty\}}(s,\pi\otimes\check{\pi})}.\nonumber
\end{equation}
Note that we choose the global measure on $Z(\A)G(\Q)\backslash G(\A)$ to be the Tamagawa measure. In particular, the absolute value is independent of $i$ and using \cite{hoffstein_lockhart} we get
\begin{equation}
	\abs{C_{\pi}^{(i)}}^{2}  \ll_{\epsilon} p^{\epsilon n_{\pi}}\cdot (1+t_{\pi})^{\epsilon}. \nonumber
\end{equation}

Combining everything we end up with
\begin{multline}
	\frac{\phi_i(n(x)a(y))}{\Vert \phi_i\Vert_2} = C_{\pi}^{(i)} \sum_{k\in \N_0}\sum_{\substack{ 0 \neq n\in\Z,\\ (n,p)=1}}\sgn(n)^{\rho}\frac{\lambda_{\pi}\left(n\right)}{\sqrt{\abs{n}}}W_p^{(i)}(a(np^{k-m_{\pi}})) \\
	\cdot W_{\infty}\left(a\left(\frac{n}{p^{m_{\pi}-k}}y\right)\right)e\left(\frac{n}{p^{m_{\pi}-k}}x\right), \nonumber
\end{multline}
for $x\in \R$ and $y\in \R^+$. Here $\rho\in\{0,1\}$ depends on whether $\phi_1,\ldots,\phi_d$ are even or odd.

Let $v_1,\ldots, v_d$ be an orthogonal basis of $\pi_p^{K_p(m_{\pi})}$. We fix a Whittaker functional and thus an embedding $$v\mapsto W_v\in\mathcal{W}(\pi_p,\psi_p).$$ Define
\begin{equation}
	\mathcal{S}_{\pi_p}(g_p) = \sum_{i=1}^d \frac{\abs{W_{v_i}(g_p)}^2}{\langle W_{v_i},W_{v_i}\rangle_{\mathcal{W}(\pi_p,\psi_p)}}.\nonumber
\end{equation}
Note that $\mathcal{S}_{\pi_p}$ is well defined as it is independent of the choice of Whittaker functional and the choice of basis $v_1,\ldots,v_d$. 

\begin{lemmy}\label{lm:first_fourier_red}
For any orthonormal basis $\phi_1,\ldots,\phi_d$ we have
\begin{equation}
	\Phi(g)\leq (dT)^{\epsilon}\sum_{k\in \N_0}\sum_{\substack{ 0 \neq n\in\Z,\\ (n,p)=1}}\frac{\abs{\lambda_{\pi}\left(n\right)}}{\sqrt{\abs{n}}} \cdot \abss{W_{\infty}\left(\frac{n}{p^{m_{\pi}-k}}y\right)} \cdot\mathcal{S}_{\pi_p}(a(np^{k-m_{\pi}}))^{\frac{1}{2}}.\nonumber
\end{equation}
where $g=n(x)a(y)\in \mathcal{F}$.
\end{lemmy}
\begin{proof}
To simplify notation we define
\begin{equation}
	a(t) = \sgn(n)^{\rho}\frac{\lambda_{\pi}\left(n\right)}{\sqrt{\abs{n}}}W_{\infty}\left(\frac{n}{p^{m_{\pi}-k}}y\right)e\left(\frac{n}{p^{m_{\pi}-k}}x\right) \text{ and }b_i(t) = W_p^{(i)}(a(np^{k-m_{\pi}})) \nonumber
\end{equation}
if $t=np^{k-m_{\pi}}$ for $k\in\N_0$ and $(n,p)=1$, and $a(t)=0=b_i(t)$ otherwise. The Whittaker expansion now neatly reads
\begin{equation}
	\phi_i(n(x)a(y)) = C_{\pi}^{(i)}\sum_{t\in\Q^{\times}}a(t)b_i(t).\nonumber
\end{equation}
With this at hand we estimate
\begin{align}
	\Phi(g) &= \left(\sum_{i=1}^d\left\vert C_{\pi}^{(i)}\sum_{t\in\Q^{\times}}a(t)b_i(t) \right\vert^2\right)^{\frac{1}{2}} \nonumber \\
	&\leq\max_i\abs{C_{\pi}^{(i)}}\cdot\left(\sum_{t_1\in\Q^{\times}}\sum_{t_2\in\Q^{\times}}a(t_1)\overline{a(t_2)}\sum_{i=1}^d b_i(t_1)\overline{b_i(t_2)}\right)^{\frac{1}{2}} \nonumber \\
	& \ll (dT)^{\epsilon} \left(\sum_{t_1\in\Q^{\times}}\sum_{t_2\in\Q^{\times}}\abs{a(t_1)a(t_2)}\left(\sum_{i=1}^d\abs{b_i(t_1)}^2\right)^{\frac{1}{2}}\left(\sum_{i=1}^d\abs{b_i(t_2)}^2\right)^{\frac{1}{2}}\right)^{\frac{1}{2}} \nonumber\\
	&=(dT)^{\epsilon}\sum_{t\in\Q^{\times}}\abs{a(t)}\left(\sum_{i=1}^d\abs{b_i(t)}^2\right)^{\frac{1}{2}}.\nonumber
\end{align}
The claim follows by inserting the definitions of $a(t)$ and $b_i(t)$.
\end{proof}

Before we can estimate this expression we need to investigate the size of the local average $\mathcal{S}_{\pi_p}(a(y))$. This is the content of the following subsection.

\subsection{Computing the local averages}

The computation of $\mathcal{S}_{\pi_p}(a(y))$ involves a case study and each case will be treated using different techniques. Finally, combining all possible cases, will lead to the bound
\begin{equation}
	\mathcal{S}_{\pi_p}(a(p^{-m_{\pi}}y)) \ll d^{1+\epsilon}\cdot \abs{y}_p. \label{eq:local_bound}
\end{equation}
See Lemma~\ref{1}, \ref{2} \text{ and }\ref{3} below.

\subsubsection{The Steinberg representation}

Let $V=\Ind_B^G(\abs{\cdot}^{\frac{1}{2}}\otimes \abs{\cdot}^{-\frac{1}{2}})$ Then we can identify $\pi=\text{St}$ with the unique irreducible generic subspace of $V$. Let $V^{\vee}=\Ind_B^G(\abs{\cdot}^{-\frac{1}{2}}\otimes \abs{\cdot}^{\frac{1}{2}})$. This is the dual space of $V$ and the invariant bilinear pairing is given by
\begin{equation}
	\langle f,f^{\vee}\rangle = \int_K f(k)f^{\vee}(k)dk.\nonumber
\end{equation}
Further $\tilde{\pi}=\text{St}$ can be identified as the unique irreducible generic sub-quotient of $V^{\vee}$.

Next we choose a basis $v_0,\ldots,v_p$ of $V^{K_p(1)}$. (In an analogous way one constructs the dual basis $v_0^{\vee},\ldots,v_p^{\vee}$ in $(V^{\vee})^{K_p(1)}$.) This is done as follows: we first construct
\begin{equation}
	v_p(g) = \Vol(B(\Z_p)K_p(1),dk)^{-\frac{1}{2}}\cdot \begin{cases} 
		\abs{\frac{a}{d}} &\text{ if }g=\left(\begin{matrix} a&b\\0&d\end{matrix}\right)k\in B(\Q_p)K_p(1),\\
		0&\text{ else.}
	\end{cases}\nonumber
\end{equation}
Further $\gamma_i=wn(i)$ for $i=0,\ldots,p-1$. For consistency of the indices we put $\gamma_p=1$ so that we can identify
\begin{equation}
	B(\Z_p)\backslash K_p/K_p(1) = \{\gamma_0,\ldots, \gamma_p  \} \nonumber
\end{equation}
via the Bruhat decomposition  of $G(\mathbb{F}_p)$. (Note that $\gamma_0=w$.) Finally define $v_i(g) = v_p(g\cdot \gamma_i^{-1})$. This is the desired basis.

Now there is an (up to scaling) unique $\psi_p$-Whittaker functional $\Lambda\colon V\to \C$ (resp. a unique $\psi_p^{-1}$-Whittaker functional $\Lambda^{\vee}\colon V^{\vee}\to\C$). As usual we set 
\begin{equation}
	W_v(g) = \Lambda(g.v) \text{ or }W_{v^{\vee}}(g) = \Lambda^{\vee}(g.v^{\vee}).\nonumber
\end{equation}
We will first consider the related average
\begin{equation}
	\mathcal{S}_V(y) = \sum_{i=0}^p W_{v_i}(a(y))W_{v_i^{\vee}}(a(y)). \nonumber
\end{equation}
Note that also this is independent of the choice of the particular basis $v_0,\ldots,v_p$ as long as one considers the corresponding dual basis of $V^{\vee}$.

We will write $\int^{st}$ for the stable integral as defined \cite[Definition~2.1]{lapid-mao_fourier}. By \cite[Lemma~4.4 and Remark~4.6]{lapid-mao_fourier} we get
\begin{align}
	W_{v_i}(a(y))W_{v_i^{\vee}}(a(y)) &= \int_{\Q_p}^{st}\langle n(x)a(y).v_i,a(y).v_i^{\vee}\rangle\psi_p(x)^{-1}dx\nonumber \\
	&= \abs{y}_p\int_{\Q_p}^{st}\langle n(x)v_i,v_i^{\vee}\rangle\psi_p(xy)^{-1}dx.\nonumber 
\end{align}
Knowing the exact shape of the $v_i$'s we can compute these integrals. First, we observe that a simple change of variables yields
\begin{align}
	\langle n(x).v_i,v_i^{\vee} \rangle &= \int_K v_p(kn(x)\gamma_i^{-1})v_p^{\vee}(k\gamma_i^{-1})dk=\int_K v_p(k\gamma_in(x)\gamma_i^{-1})v_p^{\vee}(k)dk\nonumber\\
	&= \frac{\sharp [B(\Z_p)/K_p(1)\cap B(\Z_p)]^{\frac{1}{2}}}{\Vol(K_p(1),dk)^{\frac{1}{2}}}\cdot \int_{K_p(1)} v_p(k\gamma_in(x)\gamma_i^{-1})dk.\nonumber
\end{align}

The case $i=p$ is somehow special and will be treated later. For now let us assume $0\leq i<p$. In this case we have
\begin{equation}
	\gamma_in(x)\gamma_i^{-1} = wn(i)n(x)n(-i)w^{-1}=wn(x)w^{-1}.\nonumber
\end{equation}
To take advantage of the support of $v_p$ we have to investigate
\begin{equation}
	kwn(x)w^{-1}=b\tilde{k} \in B \cdot K_p(1).\nonumber
\end{equation}
In view of the Iwahori-factorisation of $k$ we find that $n(x)\in N(\Q_p)\cap K_p(1)$ is necessary for the integral to be non-zero. Thus one gets
\begin{equation}
	\langle \pi(n(x))v_i,v_i^{\vee} \rangle =  \delta_{n\in N(p\Z_p)}.\nonumber
\end{equation} 
With this at hand it is easy to compute
\begin{align}
	\int_{\Q_p}^{st}\langle \pi(n(x))v_i,v_i^{\vee} \rangle \psi_p(yx)^{-1}dx = \int_{p\Z_p} \psi(yx)^{-1}dx = p^{-1}\delta_{y\in p^{-1}Z_p},\nonumber
\end{align}
for $0\leq i< p$.

We turn towards $i=p$, so that $\gamma_p=1$. Further we replace $y$ by $yp^{-1}$ and consider $y\in\Z_p$. Recall that every $k\in K_p(1)$ can be written as $k=t_k\overline{n}_kn_k \in B(\Z_p)N(p\Z_p)^t N(p\Z_p)$ by using the Iwahori-factorisation. We obtain
\begin{align}
	&\int_{\Q_p}^{st}\langle \pi_p(n(x))v_p,v_p^{\vee}\rangle \psi_p(yp^{-1}x)^{-1}dx \nonumber \\
	&\quad = \frac{\sharp [B(\Z_p)/K_p(1)\cap B(\Z_p)]^{\frac{1}{2}}}{\Vol(K_p(1),dk)^{\frac{1}{2}}}\cdot\int_{\Q_p}^{st}\int_{K_p(1)} v_p(kn(x))dk \psi_p(yp^{-1}x)^{-1}dx \nonumber \\
	&\quad = \frac{\sharp [B(\Z_p)/K_p(1)\cap B(\Z_p)]^{\frac{1}{2}}}{\Vol(K_p(1),dk)^{\frac{1}{2}}}\cdot \int_{K_p(1)}\int_{\Q_p}^{st} v_p(\overline{n}_kn(x)) \psi_p(yp^{-1}x)^{-1}dxdk.\nonumber
\end{align}
Note that the integrand only depends on $\overline{n}_k$. Therefore we start by discussing a suitable measure on $K_p(1)$. Indeed using the Iwahori factorisation we can write
\begin{equation}
	\int_{K_p(1)} f(k)dk = \frac{\Vol(K_p(1),dk)}{\Vol(B(\Z_p)\cap K_p(1),db)}\int_{B(\Z_p)\cap K_p(1)}p\int_{p\Z_p} f(b\overline{n}(pu))dudb.\nonumber
\end{equation}
If we write $\tilde{v}_p$ to be the re-normalisation of $v_p$ with $\tilde{v}_p(1)=1$, then we have
\begin{align}
	\int_{\Q_p}^{st}\langle \pi_p(n(x))v_p,v_p^{\vee}\rangle \psi_p(yp^{-1}x)^{-1}dx &= p\int_{\Q_p}^{st}\int_{p\Z_p} \tilde{v}_p(n(z)^tn(x)) \psi_p(yp^{-1}x)^{-1}dzdx \nonumber \\
	&= p\int_{\Q_p}^{st}\int_{p\Z_p} \tilde{v}_p\left(\left(\begin{matrix} 1&x-z^{-1} \\ z & zx\end{matrix}\right)\right) \psi_p(yp^{-1}(x-z^{-1}))^{-1}dzdx.\nonumber
\end{align}
In the last step we simply made a change of variables in the $x$-integral. A simple matrix computation shows that
\begin{equation}
	\left(\begin{matrix} 1&x-z^{-1} \\ z & zx\end{matrix}\right) = n(\star)\left(\begin{matrix} (zx)^{-1} & 0 \\ z & zx \end{matrix}\right).\nonumber
\end{equation}
Inserting this and using the transformation behaviour of $\tilde{v}_p$ one obtains
\begin{multline}
	\int_{\Q_p}^{st}\langle \pi_p(n(x))v_p,v_p^{\vee}\rangle \psi_p(yp^{-1}x)^{-1}dx \\ = p\int_{p\Z_p}\psi_p(yz^{-1}p^{-1})\frac{dz}{\abs{z}_p^2}\int_{\Q_p}^{st}\tilde{v}_p\left(\left(\begin{matrix} 1& 0\\x^{-1} & 1 \end{matrix}\right)\right)\psi_p(-xyp^{-1})\frac{dx}{\abs{x}_p^2}.\nonumber
\end{multline}
Both integrals can now be computed quite easily. Starting from the first one we obtain
\begin{multline}
	p\int_{p\Z_p}\psi_p(yz^{-1}p^{-1})\frac{dz}{\abs{z}_p^2} = \sum_{l=1}^{\infty}p^{l+1}\int_{\Z_p^{\times}}\psi_p(zyp^{-1-l})dz \nonumber \\
	= \sum_{l=1}^{v_p(y)-1}p^{l+1}(1-p^{-1})-\delta_{v_p(y)\geq 1}p^{v_p(y)} = -p\delta_{v_p(y)\geq 1}.\nonumber
\end{multline}
Turning to the other integral we find
\begin{align}
	\int_{\Q_p}^{st}\tilde{v}_p\left(\left(\begin{matrix} 1& 0\\x^{-1} & 1 \end{matrix}\right)\right)\psi_p(-xyp^{-1})\frac{dx}{\abs{x}_p^2} &= \int_{\Q_p\setminus \Z_p}^{st}\psi_p(-xyp^{-1})\frac{dx}{\abs{x}_p^2} \nonumber \\
	&= \sum_{l=1}^{\infty}p^{-l}\int_{\Z_p^{\times}}\psi_p(-xyp^{-1-l})dx \nonumber\\
	&= \sum_{l=1}^{v_p(y)-1}p^{-l}(1-p^{-1})-\delta_{v_p(y)\geq 1}p^{-v_p(y)-1} \nonumber \\
	&= \delta_{v_p(y)\geq 2}(p^{-1}-p^{-v_p(y)})-\delta_{v_p(y)\geq 1}p^{-1-v_p(y)}.\nonumber
\end{align}
In particular we have
\begin{equation}
		\int_{\Q_p}^{st}\langle \pi_p(n(x))v_0,v_0^{\vee}\rangle \psi_p(yp^{-1}x)^{-1}dx =\delta_{v_p(y)\geq 2}(p^{1-v_p(y)}-1)+\delta_{v_p(y)\geq 1}p^{-v_p(y)}. \nonumber
\end{equation}
Note that this can be negative, but for non-unitary representations there is no expectation for these integrals to be non-negative.

Combining the computations above and swapping back to $y\in p^{-1}\Z_p$ leads us to the following result.
\begin{lemmy}
In the notation above we have
\begin{equation}
	\mathcal{S}_V(y) = \abs{y}_p\left[\delta_{v_p(y)\geq -1}+\delta_{v_p(y)\geq 0} p^{-1}\abs{y}_p+\delta_{v_p(y)\geq 1}(\abs{y}_p-1)\right].\nonumber 
\end{equation}
\end{lemmy}

We will obtain the desired estimate by relating $\mathcal{S}_{\text{St}}(a(y))$ to $\mathcal{S}_V$. 

\begin{lemmy}\label{1}
For $\pi_p=\text{St}$ and $y\in \Q_p$ we have
\begin{equation}
	\mathcal{S}_{\pi_p}(a(y)) \ll \abs{y}_p.\nonumber
\end{equation}
\end{lemmy}
\begin{proof}
Recall that the definitions of $\mathcal{S}_{\pi_p}$ and $\mathcal{S}_V$ are independent of the choice of the underlying basis. Thus we can choose an orthogonal basis $w_1,\ldots,w_p$ of $\pi_p^{K_p(1)}$. Viewing $\pi$ as invariant subspace of $V$ we can assume that the $w_i$'s are in $V$. We then have 
\begin{equation}
	\mathcal{S}_{\pi}(a(y)) = \sum_{i=1}^p \frac{\abs{W_{w_i}(a(y))}^2}{\langle W_{w_i},W_{w_i}\rangle} = \sum_{i=1}^p \frac{W_{w_i}(a(y))W_{w_i^{\vee}}(a(y))}{\langle w_i,w_i^{\vee}\rangle}. \nonumber
\end{equation}
Finally if we choose $w_0^{\vee}\in (V^{\vee})^{K_p(1)}$ in the annihilator of the $w_1,\ldots,w_p$ and let $w_0\in V^{K_p(1)}$ be the dual element then after renormalising we have
\begin{equation}
	\mathcal{S}_V(y) = \mathcal{S}_{\pi}(a(y)) +  W_{w_0}(a(y))W_{w_0^{\vee}}(a(y)).\nonumber
\end{equation} 
However, since there is a unique Whittaker functional on $V^{\vee}$ which descents to the unique Whittaker functional on $\pi$ when viewed as a sub-quotient we must have $W_{w_0^{\vee}}(a(y))=0$. (Since the unique invariant subspace is non-generic.) Thus $\mathcal{S}_V(y) = \mathcal{S}_{\pi_p}(a(y))$ and the desired estimate follows directly from the previous lemma.
\end{proof}

\subsubsection{Twist minimal Principal Series}

Turning to this case we assume that $\pi_p=\chi \cdot \abs{\cdot}^{\rho}\boxplus \abs{\cdot}^{-\rho}$ where $a(\chi)=n_{\pi}>0$. Without loss of generality we can assume that $\chi(p)=1$. (If we assume that $\pi$ is unitary then it is tempered so that $\rho\in i\R$.) Now we can choose a basis in the induced picture essentially as above, but we need to find a suitable decomposition of $B(\Z_p)\backslash K_p/K_p(n_{\pi})$ (since the Bruhat decomposition does not hold in $G(\Z_p/p^m\Z_p)$ if $m>1$). First we start by defining 
\begin{equation}
	v_0(g) = \Vol(B(\Z_p)K_p(m_{\pi}),dk)^{-\frac{1}{2}}\cdot \begin{cases} 
		\chi(a)\abs{\frac{a}{d}}^{\frac{1}{2}+\rho} &\text{ if }g=\left(\begin{matrix} a&b\\0&d\end{matrix}\right)k\in B(\Q_p)K_p(m_{\pi}),\\
		0&\text{ else.}
\end{cases}\nonumber
\end{equation}
From this element we can construct a basis of $\pi_p^{K_p(m_{\pi})}$ as in the Steinberg case. Indeed, we fix a system of representatives $\{\gamma_j\}$ for $B(\Z_p)\backslash K_p/K_p(m_{\pi})$ and set $v_j=\pi_p(\gamma_j^{-1})v_0$.

In order to explicate this basis we need to compute a suitable coset decomposition for $B(\Z_p)\backslash K_p/K_p(m_{\pi})$. This is the content of the following lemma.

\begin{lemmy}
We have
\begin{equation}
	K_p = \bigsqcup_{a\in \Z_p/p^{m_{\pi}}\Z_p} B(\Z_p) \gamma_{0,a} K_p(m_{\pi}) \sqcup \bigsqcup_{i=1}^{m_{\pi}}\bigsqcup_{a\in (\Z_p/p^{m_{\pi}-i}\Z_p)^{\times}} B(\Z_p)\gamma_{i,a} K_p(m_{\pi}).\nonumber
\end{equation}
for 
\begin{equation}
	\gamma_{i,a} = \begin{cases}
		\left(\begin{matrix} 0& -1 \\ 1 & a\end{matrix} \right) & \text{ if }i=0,\\
		\left(\begin{matrix} 1 & 0 \\ ap^i &1 \end{matrix} \right) &\text{ if } 1\leq i\leq m_{\pi}-1, \\
		1_2 &\text{ if }i=m_{\pi}.
	\end{cases} \nonumber
\end{equation}
\end{lemmy}
\begin{proof}
Take $g=\left(\begin{matrix} a&b\\c & d\end{matrix}\right)\in K_p$ and set $i=v_p(c)$. We treat several cases distinguished by the value of $i$.

First, if $i=m_{\pi}$, then we have
\begin{equation}
	g=\left(\begin{matrix} a & b\\ c & d\end{matrix}\right) = \left( \begin{matrix} a&b\\0&d \end{matrix}\right) \left(\begin{matrix} 1-\frac{bc}{ad} & 0 \\ \frac{c}{d} & 1 \end{matrix}\right) \in B(\Z_p)K_p(m_{\pi}).\nonumber
\end{equation}

Second, for $1\leq i<m_{\pi}$ we have
\begin{equation}
	g=\left( \begin{matrix} a-\frac{bc}{d} & b \\ 0&d\end{matrix}\right) \left( \begin{matrix} 1 & 0 \\ \frac{c}{d} & 1 \end{matrix}\right).\nonumber
\end{equation}
By right multiplication with elements in $K_p(m_{\pi})$ we can view $\frac{c}{d} \in p^i\Z_p^{\times}/p^{m_{\pi}}\Z_p$.

The critical contribution is given by the matrices with $i=0$. We can write
\begin{equation}
	g= \left(\begin{matrix} \frac{ad}{c}-b & a+(b-\frac{ad}{c})p^{m_{\pi}}\\ 0 & c \end{matrix}\right)\left(\begin{matrix}0 & -1 \\ 1 & \frac{d}{c} \end{matrix}\right)\left(\begin{matrix} 1+\frac{d}{c}p^{m_{\pi}} & \frac{d^2}{c^2}p^{m_{\pi}} \\ -p^{m_{\pi}} & 1-\frac{d}{c}p^{m_{\pi}} \end{matrix}\right). \nonumber
\end{equation}
\end{proof}

Given $v \in \pi^{K_p(n_{\pi})}$ we can compute the Jacquet Integral as follows. Without loss of generality assume $v_p(y)\geq -n_{\pi}$, since otherwise the Whittaker function $W_v$ vanishes for trivial reasons. We compute
\begin{align}
	W_v(a(y)) &= \int_{\Q_p} v(wn(x)a(y))\psi_p(x)^{-1}dx \nonumber \\
	&= \abs{y}_p^{\frac{1}{2}-\rho}\int_{\Q_p}v(wn(x))\psi_p(xy)^{-1}dx\nonumber \\
	&= \abs{y}_p^{\frac{1}{2}-\rho} p^{-n_{\pi}}\sum_{a\in \Z_p/p^{n_{\pi}}\Z_p} \psi_p(ay)^{-1}v(wn(a)) \nonumber \\
	&\qquad  + \abs{y}_p^{\frac{1}{2}-\rho}\int_{\Q_p\setminus \Z_p}\abs{x}^{-2\rho} \chi(x)^{-1}\psi(xy)^{-1}v(n(x^{-1})^t)\frac{dx}{\abs{x}_p}.\nonumber
\end{align}
Note that $wn(a) = \gamma_{0,a}$. Now we will have a closer look at the remaining integral:
\begin{align}
	&\int_{\Q_p\setminus \Z_p}\abs{x}^{-2\rho} \chi(x)^{-1}\psi(xy)^{-1}v(n(x^{-1})^t)\frac{dx}{\abs{x}_p} \nonumber \\
	&\qquad = \int_{p\Z_p}\abs{x}^{2\rho} \chi(x)\psi(x^{-1}y)^{-1}v(n(x)^t)\frac{dx}{\abs{x}_p} \nonumber\\
	&\qquad = v(1)\int_{p^{m_{\pi}}\Z_p}\chi(x)\abs{x}^{2\rho}\psi(x^{-1}y)^{-1}\frac{dx}{\abs{x}_p} \nonumber \\
	&\qquad\qquad +\sum_{i=1}^{m_{\pi}-1}\sum_{b\in (\Z_p/p^{m_{\pi}-i}\Z_p)^{\times}}p^{-2\rho i}\chi(b)^{-1}v(n(bp^i)^t)\cdot \int_{1+p^{m_{\pi}-i}\Z_p}\chi(x)^{-1}\psi(-byxp^{-i})dx.\nonumber
\end{align}
Note that $n(bp^i)^t=\gamma_{i,b}$.

Since the $v(1)$-contribution is easily computed we arrive at the following lemma.
\begin{lemmy}
For $v_p(y)\geq -m_{\pi}$ we have
\begin{align}
	W_v(a(y)) &= \abs{y}_p^{\frac{1}{2}-\rho}p^{-m_{\pi}}\sum_{a\in \Z_p/p^{m_{\pi}}\Z_p}\psi_p(ay)^{-1}v(\gamma_{0,a}) \nonumber \\
	&\qquad + \abs{y}_p^{\frac{1}{2}-\rho}\sum_{i=1}^{m_{\pi}-1}\sum_{b\in (\Z_p/p^{m_{\pi}-i}\Z_p)^{\times}}p^{-2\rho i}\chi(b)^{-1}v(\gamma_{i,b})G_{m_{\pi}-i}(-byp^{-i},\chi^{-1}) \nonumber\\
	&\qquad +\abs{y}_p^{\frac{1}{2}-\rho}\cdot  v(\gamma_{m_{\pi},0})\cdot \int_{p^{m_{\pi}}\Z_p}\chi(x)\abs{x}^{2\rho}\psi(x^{-1}y)^{-1}\frac{dx}{\abs{x}_p}. \nonumber
\end{align}
for
\begin{equation}
	G_l(y,\chi) = \int_{1+p^l\Z_p}\chi(x)\psi(yx)dx.\nonumber
\end{equation}
\end{lemmy}

This supplies us with the necessary ingredients to show the required estimate for $\mathcal{S}_{\pi_p}$.
\begin{lemmy}\label{2}
For $\pi_p=\chi\abs{\cdot}^{\rho}\boxplus \abs{\cdot}^{-\rho}$ unitary and $y\in \Q_p$ we have 
\begin{equation}
	\mathcal{S}_{\pi_p}(a(p^{-m_{\pi}}y)) \ll d^{1+\epsilon}\abs{y}_p. \nonumber
\end{equation}
\end{lemmy}
\begin{proof}
Note that since all $v_j$'s are translates of $v_0$ their Whittaker-norm all coincides. So it suffices to compute one of these norms and it is easy to see that 
\begin{equation}
	\langle W_{v_0},W_{v_0}\rangle = \int_{\Q_p^{\times}}\abs{W_{v_0}(a(y))}^2\frac{dy}{\abs{y}} = (1+p^{-1}). \nonumber
\end{equation}

Next we observe that one can choose representatives so that $v_j=\pi_p(\gamma_{i,a}^{-1})v_0$ for some $i=i(j)$ and $a=a(j)$. In particular, we can sort the terms of the sum $\mathcal{S}_{\pi_p}(a(y))$ according to this $i$. We get
\begin{equation}
	\mathcal{S}_{\pi_p}(a(y)) = \sum_{i=0}^{m_{\pi}} \mathcal{S}_i(y) \text{ for } \mathcal{S}_i(y) = (1+p^{-1})^{-1}\sum_{a\in \Z_p/p^{m_{\pi}-i}\Z_p} \abs{W_{\pi_p(\gamma_{i,a}^{-1})v_0}(a(y))}^2.\nonumber
\end{equation}
Applying the previous Lemma with $v=\pi_p(\gamma_{i,a}^{-1})v_0$ and taking support properties of $v_0$ into account provides us with nice formulae for the $W_{\pi_p(\gamma_{i,a}^{-1})v_0}(a(y))$. 

As soon as we can show that $\mathcal{S}_i(y)\ll \abs{y}_p$ for all $i$ we are done. We start with $i=0$. Here we have the explicit formula
\begin{equation}
	\mathcal{S}_i(y) = (1+p^{-1})^{-1}\sum_{a\in \Z_p/p^{m_{\pi}}\Z_p} p^{-2m_{\pi}}\abs{y}_p v_0(1)^2 = \frac{(1+p^{-1})^{-1}p^{-m_{\pi}}\abs{y}_p}{ \Vol(B(\Z_p)K_p(m_{\pi}),dk)} = \abs{y}_p\nonumber
\end{equation}
for $y\in p^{-m_{\pi}}\Z_p$. 

We turn towards $1\leq i \leq m_{\pi}-1$. In this range we get
\begin{align}
	\mathcal{S}_i(y) &= (1+p^{-1})^{-1}v_0(1)^2\abs{y}_p \sum_{a\in (\Z_p/p^{m_{\pi}-i}\Z_p)^{\times}}\abs{G_{m_{\pi}-i}(-ayp^{-i},\chi^{-1})}^2.\nonumber
\end{align}
Thus we need to bound the integrals $G_l(z,\chi)$, which are somehow incomplete Gau\ss\  sums in the sense that one sums only over a specific congruence class. These sums were essentially computed in the proof of \cite[Lemma~5.8]{assing}. Indeed one extracts
\begin{equation}
	G_l(zp^{-k},\chi) = \begin{cases}
		\epsilon(\frac{1}{2},\chi^{-1})\chi^{-1}(z)p^{-\frac{k}{2}} &\text{ if }l\leq \lfloor \frac{a(\chi)}{2}\rfloor, k=a(\chi) \text{ and }z\in b(\chi)+p^l\Z_p,\\
		\psi_p(zp^k)p^{-l} &\text{ if }l\geq \lceil \frac{a(\chi)}{2}\rceil, k=a(\chi) \text{ and }z\in -b(\chi)+p^{a(\chi)-l}\Z_p,\\
		0&\text{ else.}
	\end{cases} \nonumber
\end{equation}
for $z\in \Z_p^{\times}$, $k\in \Z$ and $b(\chi)\in \Z_p^{\times}$ is determined by $\chi$. With this at hand we can easily evaluate $\mathcal{S}_i$. For $i\leq \lfloor \frac{m_{\pi}}{2}\rfloor$ we have
\begin{align}
	\mathcal{S}_i(y) &= \delta_{v_p(y) =i-m_{\pi}}(1+p^{-1})^{-1}v_0(1)^2\abs{y}_p \sum_{b\in (\Z_p/p^{m_{\pi}-i}\Z_p)^{\times}}p^{2i-2m_{\pi}}\delta_{-byp^{-v_p(y)}\in -b(\chi)+p^i\Z_p} \nonumber\\
	&= \delta_{v_p(y) =i-m_{\pi}}\abs{y}_p.\nonumber
\end{align}
Similarly for $i\geq \lceil \frac{m_{\pi}}{2}\rceil$ we have
\begin{align}
	\mathcal{S}_i(y) &= \delta_{v_p(y) =i-m_{\pi}}(1+p^{-1})^{-1}v_0(1)^2\abs{y}_p \sum_{b\in (\Z_p/p^{m_{\pi}-i}\Z_p)^{\times}}p^{-m_{\pi}}\cdot \delta_{-byp^{-v_p(y)}\in b(\chi)+p^{m_{\pi}-i}\Z_p} \nonumber\\
	&= \delta_{v_p(y) =i-m_{\pi}}\abs{y}_p.\nonumber
\end{align}

Finally consider $i=m_{\pi}$. We have
\begin{align}
	\mathcal{S}_{m_{\pi}}(y) &= (1+p^{-1})v_0(1)^2\abs{y}_p \abs{\int_{p^{m_{\pi}}\Z}\chi(x)\abs{x}^{2\rho}\psi(x^{-1}y)^{-1}\frac{dx}{\abs{x}_p}}^2 \nonumber \\
	&= p^{m_{\pi}}\abs{y}_p\cdot \abs{\int_{p^{m_{\pi}}\Z_p}\chi(x)\abs{x}^{2\rho}\psi(x^{-1}y)^{-1}\frac{dx}{\abs{x}_p}}^2.\nonumber
\end{align}
Therefore it suffices to compute the remaining integral. By some basic Gau\ss\  sum evaluations one gets
\begin{equation}
	\int_{p^{m_{\pi}}\Z_p}\chi(x)\abs{x}^{2\rho}\psi(x^{-1}y)^{-1}\frac{dx}{\abs{x}_p} = \delta_{y\in \Z_p}\epsilon(\frac{1}{2},\chi)\chi(y)p^{-\frac{m_{\pi}}{2}-2\rho[v_p(y)+m_{\pi}]}. \nonumber
\end{equation}
Inserting this above concludes the proof since it implies $\mathcal{S}_{m_{\pi}}(y) =\delta_{v_p(y)>0}\abs{y}_p$.
\end{proof}

\subsubsection{Supercuspidal representations}

Let $\mathfrak{X}_k$ be the set of character $\chi\colon \Q_p^{\times} \to\C^{\times}$ with $a(\chi)\leq k$ and $\chi(p)=1$. Note that
\begin{equation}
	\sharp \mathfrak{X}_k = p^{k-1}(p-1).\nonumber
\end{equation}
For $\chi\in \mathfrak{X}_k$ and $m\in \Z$ we will consider the functions $\xi_{\chi}^{(m)} \in \mathcal{C}_c^{\infty}(\Q_p^{\times})$ given by
\begin{equation}
	\xi_{\chi}^{(m)}(y) = \mathbbm{1}_{p^{-m}\Z_p^{\times}}(y)\chi(y).\nonumber
\end{equation}
Given any representation $\pi_p$ we write $\mathcal{K}_{\psi_p}(\pi_p)$ for the corresponding $\psi_p$-Kirillov model. Note that this model contains the Schwartz functions so that we have $\xi_{\chi}^{(m)}\in \mathcal{K}_{\psi_p}(\pi_p)$. Note that by construction of the Kirillov model we have
\begin{equation}
	W_f(a(y))=f(y) \text{ for } f\in\mathcal{K}_{\psi_p}(\pi_p) \text{ and }y\in \Q_p^{\times}.\nonumber
\end{equation}
Thus we compute
\begin{multline}
	\langle W_{\xi_{\chi_1}^{(m_1)}},W_{\xi_{\chi_2}^{(m_2)}}\rangle = \int_{\Q_p^{\times}}\xi_{\chi_1}^{(m_1)}(y)\overline{\xi_{\chi_2}^{(m_2)}(y)}d^{\times}y = \delta_{m_1=m_2} \int_{p^{-m_1}\Z_p^{\times}}\chi_1(y)\chi_2^{-1}(y)d^{\times}y\\ = \delta_{\substack{m_1=m_2,\\ \chi_1=\chi_2}}.\nonumber
\end{multline}
This suffices to compute $S_{\pi_p}(a(y))$ for supercuspdial representations $\pi_p$.

\begin{lemmy}\label{3}
Suppose $\pi_p$ is a twist minimal supercuspidal representation with (exponent)-conductor $n_{\pi}$. Then the following is true
\begin{itemize}
	\item If $n_{\pi}=2m_{\pi}$, then
	\begin{equation}
		S_{\pi_p}(a(y)) = \frac{p^{m_{\pi}}}{\zeta_p(1)}\cdot \delta_{y\in p^{-m_{\pi}}\Z_p^{\times}}.\nonumber
	\end{equation}
	\item If $n_{\pi} = 2m_{\pi}-1$, then 
	\begin{equation}
		S_{\pi_p}(a(y)) = \frac{p^{m_{\pi}}}{\zeta_p(1)}\cdot \delta_{y\in p^{-m_{\pi}}\Z_p^{\times}}+\frac{p^{m_{\pi}-1}}{\zeta_p(1)}\cdot \delta_{y\in p^{1-m_{\pi}}\Z_p^{\times}}.\nonumber
	\end{equation}
\end{itemize}
In general we have the bound $$S_{\pi_p}(a(p^{-m_{\pi}}y)) \ll d\cdot\abs{y}_p\cdot \delta_{y\in \Z_p}.$$
\end{lemmy}
\begin{proof}
We start with the case $n_{\pi} = 2m_{\pi}$. By \cite[Lemma~4.4]{miyauchi-yamauchi_remark} we find that a basis for $\pi^{K_p(m_{\pi})}$ in the Kirillov model is given by
\begin{equation}
	\{ \xi_{\chi}^{(m_{\pi})}\colon \chi\in \mathfrak{X}_{m_{\pi}}\}.\nonumber
\end{equation}
Note that we already took advantage of twist-minimality using that $n_{\chi\pi} =n_{\pi}$ for all $\chi\in \mathfrak{X}_{m_{\pi}}$. Our computations above show that this basis is orthonormal (with respect to the Whittaker inner product). Thus we have
\begin{equation}
	\mathcal{S}_{\pi_p}(y) = \sum_{\chi\in \mathfrak{X}_{m_{\pi}}} \abs{\xi_{\chi}^{(m_{\pi})}(y)}^2 = \delta_{y\in p^{-m_{\pi}}\Z_p^{\times}}\cdot \sharp \mathfrak{X}_{m_{\pi}}.\nonumber
\end{equation}

We turn towards the second case where $n_{\pi}$ is odd. Then we get the orthonormal basis
\begin{equation}
	\{\xi_{\chi}^{m_{\pi}}\colon \chi \in \mathfrak{X}_{m_{\pi}}\} \cup \{\xi_{\chi}^{m_{\pi}-1}\colon \chi \in \mathfrak{X}_{m_{\pi}-1}\}.\nonumber
\end{equation}
It is again easy to compute the desired quantity:
\begin{equation}
	\mathcal{S}_{\pi_p}(y) = \delta_{y\in p^{-m_{\pi}}\Z_p^{\times}}\cdot \sharp \mathfrak{X}_{m_{\pi}}+\delta_{y\in p^{1-m_{\pi}}\Z_p^{\times}}\cdot \sharp \mathfrak{X}_{m_{\pi}-1}.\nonumber
\end{equation}
The result follows directly.
\end{proof}

\subsection{Conclusion}

We can now give a decent bound for $\Phi(z)$ using the Whittaker expansion. We will use the bound \eqref{eq:local_bound} and follow the standard procedure.

\begin{lemmy}\label{lm:whitt_bound}
We have
\begin{equation}
	\Phi(z) \ll \left(\frac{dT}{y}\right)^{\epsilon}\left( d^{\frac{1}{2}}T^{\frac{1}{6}}+\frac{dT^{\frac{1}{2}}}{y^{\frac{1}{2}}}\right).\nonumber
\end{equation}
\end{lemmy}
\begin{proof}
Inserting \eqref{eq:local_bound} into Lemma~\ref{lm:first_fourier_red} yields
\begin{equation}
	\Phi(g)\ll d^{\frac{1}{2}+\epsilon}T^{\epsilon}\sum_{0 \neq n\in\Z}\frac{\abs{\lambda_{\pi}\left(n/(n,p^{\infty})\right)}}{\sqrt{\abs{n}}} \abs{W_{\infty}\left(n\frac{y}{p^{m_{\pi}}}\right)}.\nonumber
\end{equation}
Estimating the remaining $n$-sum as for example in \cite{templier_hybrid} or \cite{saha} yields the desired result.
\end{proof}

\section{A bound via the pre-trace formula}

The next bound will be derived from the pre-trace inequality. We start by discussing the local test functions. At the archimedean place we closely follow \cite[Section~3.5]{saha} and fix $f_{\infty}$ so that it satisfies
\begin{enumerate}
	\item $f_{\infty}(g) = 0$ unless $g\in G(\R)^+$ and $u(g)\leq 1$;
	\item $\hat{f}_{\infty}(\sigma)>0$ for all irreducible spherical unitary principal series representations $\sigma$ of $G(\R)$;
	\item $\hat{f}(\pi_{\infty})\gg 1$;
	\item $\abs{f_{\infty}(g)}\leq T$ and if $u(g)\geq T^{-2}$, then $\abs{f_{\infty}(g)}\leq T^{\frac{1}{2}}u(g)^{-\frac{1}{4}}$.
\end{enumerate}
(The final property is not really necessary because we are ignoring the spectral aspect for now.) Note that $\hat{f}$ is the spherical transform (also Selberg/Harish-Chandra transform) of $f$ and $u(g)$ is the point-pair invariant on group level. 

At the place $v=p$ we define multiple test functions:
\begin{equation}
	f_p^{(i)}(g) = \mathbbm{1}_{ZK}(g) \frac{\overline{\langle \pi(g)\phi_p^{(i)},\phi_p^{(i)}\rangle_{\pi_p}}}{\langle \phi_p^{(i)},\phi_p^{(i)}\rangle_{\pi_p}}.\nonumber
\end{equation}

\begin{lemmy}
For every irreducible admissible unitary representation $\sigma$ of $G(\Q_p)$ the operator $\sigma(f_p^{(i)})$ is non-negative and self-adjoint. Furthermore we have $\pi_p(f_p^{(i)})\phi_p^{(i)} = \dim_{\C}(\pi_p^{K_p(m_{\pi})})^{-1}\cdot \phi_p^{(i)}$.
\end{lemmy}
\begin{proof}
The proof is standard and relies on Schur's orthogonality relations for irreducible representations of $K_p$. To apply this it will be important to keep in mind that $\pi^{K_{p}(m_{\pi})}$ is irreducible.

The operators are self-adjoint since $f_p^{(i)}(g^{-1})=\overline{f_p^{(i)}(g)}$. To see non-negativity we will show the convolution identity
\begin{equation}
	f_p^{(i)} = \dim_{\C}\pi^{K_{p,1}(m_{\pi})} \cdot (f_p^{(i)}\star f_p^{(i)}).\nonumber
\end{equation}
Indeed we compute
\begin{align}
	[f_p^{(i)}\star f_p^{(i)}](h) &=\int_{Z\backslash G} f_p^{(i)}(g^{-1})f_p^{(i)}(gh)dg \nonumber \\
	&= \frac{\mathbbm{1}_{ZK}(h)}{{\langle \phi_p^{(i)},\phi_p^{(i)}\rangle_{\pi_p}}^2} \int_{K_p}\langle \pi_p(g)\phi_p^{(i)},\phi_p^{(i)}\rangle_{\pi_p}\overline{\langle \pi_p(gh)\phi_p^{(i)},\phi_p^{(i)}\rangle_{\pi_p}}dk\nonumber  \\
	&= \dim_{\C}(\pi^{K_{p}(m_{\pi})})^{-1}\frac{\mathbbm{1}_{ZK}(h)}{\langle \phi_p^{(i)},\phi_p^{(i)}\rangle_{\pi_p}}\overline{\langle \pi_p(h)\phi_p^{(i)},\phi_p^{(i)}\rangle_{\pi_p}}\nonumber
\end{align}
and the claimed identity follows directly.

It remains to show the final claim. First observe that the image of $\pi_p(f_p^{(i)})\phi_p^{(i)}$ is obviously $K_p(m_{\pi})$-invariant. Thus it suffices to show that $$\langle\pi_p(f_p^{(i)})\phi_p^{(i)},w\rangle_{\pi_p} = \dim_{\C}(\pi_p^{K_p(m_{\pi})})^{-1}\cdot\langle\phi_p^{(i)},w\rangle_{\pi_p}$$ for any $w\in \pi_p^{K_p(m_{\pi})}$ But this follows again from the orthogonality relations since $\pi_p^{K_p(m_{\pi})}$ is irreducible (as $K_p$-module) and
\begin{equation}
	\langle\pi_p(f_p^{(i)})\phi_p^{(i)},w\rangle_{\pi_p} = \langle \int_{Z\backslash G} f_p^{(i)}(g)\pi_p(g)\phi_p^{(i)}dg,w\rangle =\int_{Z\backslash G}f_p^{(i)}(g)\langle \pi_p(g)\phi_p^{(i)},w\rangle_{\pi_p}dg.\nonumber
\end{equation}
\end{proof}

Finally we define the unramified part of the test function $f_{ur}$ by setting 
\begin{multline}
	f_{ur} = \left( \sum_{l\in S} c_l \kappa_l\right)\star \left( \sum_{l\in S} c_l \kappa_l \right)^*+\left( \sum_{l\in S} c_{l^2} \kappa_{l^2}\right)\star \left( \sum_{l\in S} c_{l^2} \kappa_{l^2} \right)^*, \\ \text{ for } c_r=\begin{cases} \frac{\abs{\lambda_{\pi}(r)}}{\lambda_{\pi}(r)} &\text{ if $r=l$ or $r=l^2$ for $l\in S$},\\ 0&\text{ else.}\end{cases} \nonumber
\end{multline}
for a set of primes $S$ (to be determined) and normalised $r$th Hecke-operators $\kappa_r$. This implements the usual amplification procedure. Finally we define the global test functions
\begin{equation}
	f^{(i)} = f_{\infty}\otimes f_p^{(i)}\otimes f_{ur} \text{ and } f=\sum_i f^{(i)}.\nonumber
\end{equation}

We introduce
\begin{equation}
	M(l, g) = \{A\in M_2(\Z) \colon \det(A)=l,\, A\equiv g \text{ mod }p^{m_{\pi}} \} \text{ for }g\in \GL_2(\Z/p^{m_{\pi}}\Z).\nonumber
\end{equation}
Further let $\sigma$ denote the irreducible representation of $\GL_2(\Z/p^{m_{\pi}}\Z)$ through which the irreducible $K_p$-module $\pi^{K_p(m_{\pi})}$ factors. This is a representation of a finite group and we write $\chi_{\sigma}$ for its character. Finally we define the coefficients $y_r$ by linearising the convolutions of Hecke-operators in the definition of $f_{ur}$. More precisely we write
\begin{equation}
	f_{ur} = \sum_{r} y_r \kappa_r.\nonumber
\end{equation}
This can be compared to the analogous expression in \cite[Section~7]{saha}.

The following pre-trace inequality provides the transition to the counting problem.

\begin{lemmy}
For $z\in \mathcal{F}$ we have
\begin{equation}
	\frac{(\sharp S)^2}{\dim_{\C}\pi_p^{K_p(m_{\pi})}} \cdot \Phi(z)^2 \ll  \sum_{r} \frac{\abs{y_r}}{\sqrt{r}} \sum_{g\in \GL_2(\Z/p^{m_{\pi}}\Z)} \abs{\chi_{\sigma}(g)} \sum_{A\in M(l,g)} \abs{f_{\infty}(u(Az,z))}.\nonumber
\end{equation}
\end{lemmy}
\begin{proof}
We start by considering the spectral expansion of the automorphic kernel $k_{f^{(i)}}$ associated to the self-adjoint operators $R(f^{(i)})$ and dropping all terms except $\phi_i$. The latter is possible by positivity. We obtain
\begin{equation}
	\frac{(\sharp S)^2}{\dim_{\C}\pi_p^{K_p(m_{\pi})}} \cdot \phi_i(g_{\infty})^2 \leq k_{f^{(i)}}(g,g) = \sum_{r}y_r \sum_{\gamma\in Z(\Q)\backslash G(\Q)} f_p^{(i)}(\gamma)\kappa_r(\gamma)f_{\infty}(g_{\infty}^{-1}\gamma g_{\infty}).\nonumber
\end{equation}
We now sum this inequality over $i$ to obtain
\begin{equation}
	\frac{(\sharp S)^2}{\dim_{\C}\pi_p^{K_p(m_{\pi})}}\Phi(g_{\infty})^2 \leq \sum_r y_r \sum_{\gamma\in Z(\Q)\backslash G(\Q)} \left(\sum_{i=1}^d f_p^{(i)}(\gamma)\right) \kappa_r(\gamma)f_{\infty}(g_{\infty}^{-1}\gamma g_{\infty}).\nonumber
\end{equation}
Now we write $\overline{\gamma}$ for the image of $\gamma$ in $\GL_2(\Z/p^{m_{\pi}}\Z)$. Note that this is well defined as long as $\gamma\in K_p$. For such $\gamma$ we get
\begin{equation}
	\sum_{i=1}^d f_p^{(i)}(\gamma) = \chi_{\sigma}(\overline{\gamma}).\nonumber
\end{equation}
The rest of the argument is standard and can for example be found in \cite{saha}.
\end{proof}

By the choice of $f_{\infty}$ we can already eliminate the archimedean  influence from the right hand side. (Note that we are not aiming to amplify in the $T$-aspect.)

\begin{cor}\label{pre-trace_inequality}
For $z\in \mathcal{F}$ we have
\begin{equation}
	\frac{(\sharp S)^2}{\dim_{\C}\pi_p^{K_p(m_{\pi})}} \cdot \Phi(z)^2 \ll  T \sum_{g\in \GL_2(\Z/p^{m_{\pi}}\Z)} \abs{\chi_{\sigma}(g)} \sum_{r} \frac{\abs{y_r}}{\sqrt{r}}  \cdot \sharp M_z(r,g).\nonumber
\end{equation}
for
\begin{equation}
	M_z(r, g) = \{A\in M_2(\Z) \colon \det(A)=r,\, A\equiv g \text{ mod }p^{m_{\pi}} \text{ and }u(Az,z)\leq 1 \}.\nonumber
\end{equation}
\end{cor}

This last corollary tells us that we need to control the character $\chi_{\sigma}$ and solve a counting problem estimating $M_z(r,g)$.

To estimate the character we need to define certain level sets.
\begin{equation}
	K_{m,\lambda} = \{g\in G(\Z/p^m\Z)\colon g\equiv 1 \text{ mod }p^{\lambda} \} \text{ for } 0\leq \lambda \leq m.\nonumber
\end{equation}
Note that $K_{m,m} = \{1\}$ and $K_{m,0} = G(\Z/p^m\Z)$. 

\begin{lemmy}\label{character}
Suppose $p>3$. Let $\pi$ belong to Case~1, ~2 or ~3.1. Then, for $0\leq \lambda < m_{\pi}$ and $g\in Z\cdot K_{m_{\pi},\lambda}\setminus Z\cdot K_{m_{\pi},\lambda+1}$  we have
\begin{equation}
	\chi_{\sigma}(g) \ll p^{\lambda}.
\end{equation}
If $\pi$ belongs to Case~3.2 and $\lambda$ and $g$ are as above, then we have the slightly weaker bound	
\begin{equation}
	\chi_{\sigma}(g) \ll p^{\frac{m_{\pi}+\lambda}{2}} .\nonumber
\end{equation}
Furthermore, let $h\in G(\Z/p^m\Z)$ be a diagonal matrix such that $G(\Z/p^m\Z)=Z\cdot \SL_2(\Z/p^m\Z)\sqcup  hZ\cdot \SL_2(\Z/p^m\Z)$. Then the same estimates hold for $hZ \cdot K_{m_{\pi},\lambda}\setminus hZ\cdot K_{m_{\pi},\lambda+1}$, where $\det(h)$ is not a square modulo $p^{m_{\pi}}$.
\end{lemmy}
The representations of $\GL_2$ over finite rings such as $\Z/p^m\Z$ and their characters are well studied but explicit estimates for the characters as needed here seem to be hard to find. We choose to use the character tables for $\SL_2(\Z/p^m\Z)$ computed by Kutzko in his PhD thesis. This makes it necessary to pass from $\SL_2$ to $\GL_2$ using Mackey Theory. Note that the character values in question were calculated in \cite{leigh_cliff_wen}. However, they remain hard to extract and we hope our approach is more transparent.

\begin{proof}
Note that if $m_{\pi}=1$, then $\sigma$ is the character of $\GL_2(\Z/p\Z)$. If we are in Case~1 or ~2, then the representation is constructed by parabolic induction and the character is easily computed. Otherwise we must be in Case~3.1 in which case $\sigma$ is cuspidal. In this case the character values are well known, see for example \cite{green_char}. (Alternatively one can use Mackey Theory to reduce to the case of characters for $\SL_2(\Z/p\Z)$ and use the corresponding character table given in \cite[p. 128]{schur}.)
	
We will now assume $m_{\pi}>1$. Since Case~2 can not occur we treat Case~1 and ~3.1, leaving Case~3.2 for later. Our approach is based on reduction to the case of characters for $\SL_2(\Z/p^m\Z)$ using Mackey-Theory. Recall that we are assuming $p$ to be odd. Let $\omega_{\sigma}$ be the central character of $\sigma$ and $\tilde{\sigma}$ be an irreducible component of $\sigma\vert_{\SL_2(\Z)}$. Fix $h$ such that $\GL_2(\Z/p^m\Z) = Z\SL_2(\Z/p^m\Z)\cup h\cdot Z\SL_2(\Z/p^m\Z)$ and write $\tilde{\sigma}^h(g) = \tilde{\sigma}(hgh^{-1})$. If $\tilde{\sigma}\not\cong\tilde{\sigma}^h$, then
\begin{equation}
	\sigma  = \Ind_{Z\cdot \SL_2(\Z/p^m\Z)}^{\GL_2(\Z/p^m\Z)}(\omega_{\sigma}\cdot \tilde{\sigma}).\nonumber
\end{equation} 
Otherwise, if $\tilde{\sigma}\cong\tilde{\sigma}^h$, then
\begin{equation}
	\sigma\oplus \sigma'  = \Ind_{Z\cdot \SL_2(\Z/p^m\Z)}^{\GL_2(\Z/p^m\Z)}(\omega_{\sigma}\cdot \tilde{\sigma}),\label{eq:seecond_case}
\end{equation} 
where $\sigma'$ is another irreducible representation of $\GL_2(\Z/p^m_{\pi}\Z)$. Now we observe  that the dimensions of irreducible representations of $\SL_2(\Z/p^{m_{\pi}}\Z)$ are given by $p^{m_{\pi}}(1+p^{-1})$, $p^{m_{\pi}}(1-p^{-1})$ and $\frac{1}{2}p^{m_{\pi}}(1-p^{-2})$. Thus by recalling that $\sigma$ has dimension $p^{m_{\pi}}(1\pm p^{-1})$ (in Cases~1 and ~3.1) we find that we must be in the situation described in \eqref{eq:seecond_case}. Moreover, this means that the restriction of $\sigma$ to $Z\cdot \SL_2(\Z/p\Z)$ is irreducible and equivalent to $\omega_{\sigma\cdot\tilde{\sigma}}$. In particular the character $\chi_{\sigma}$ is given by $\chi_{\sigma}(z\cdot s)=\omega_{\sigma}(z)\chi_{\tilde{\sigma}}(s)$ for $z\in Z$ and $s\in \SL_2(\Z/p^m\Z)$. We conclude by referring to \cite[Table~III]{kutzko} where the character values of characters of dimensions $p^{m_{\pi}}(1\pm p^{-1})$ are listed. 

We now turn towards Case~3.2. Note that this case is exceptional in the sense that the restriction of $\sigma$ to $\SL_2(\Z/p^{m_{\pi}}\Z)$ is reducible. More precisely $\sigma\vert_{\SL_2(\Z/p\Z)}\cong\tilde{\sigma}\oplus \tilde{\sigma}^h$ As a consequence the character $\chi_{\sigma}$ can only be described by a combination of two character of $\SL_2(\Z/p^{m_{\pi}}\Z)$. Indeed, $\chi_{\sigma}(zs)=\omega_{\sigma}(z)[\chi_{\tilde{\sigma}}(s)+\chi_{\tilde{\sigma}^h}(s)]$. The corresponding character values are listed in \cite[Table~IV]{kutzko} and the claimed bound is derived directly by ignoring any possible cancellation between the two characters $\chi_{\tilde{\sigma}}$ and $\chi_{\tilde{\sigma}^h}$. (Even though such cancellation can be observed in the $m_{\pi}=1$ situation this phenomenon does not seem to generalise.)

If $g\in hZ\cdot \SL_2(\Z/p^m\Z)$, then the argument proceeds similarly and we omit the details.

\end{proof}

Before continuing we will discuss our choice of $S$. But first recall that $d=\dim_{\C}\pi_p^{K_p(m_{\pi})} \asymp p^{m_{\pi}}$. Further note that $y_r=0$ unless $r=1,l_1,l_1l_2,l_1^2l_2^2$ for $l_1,l_2\in S$. Put $\Lambda = d^{\frac{1}{3}}$ or $\Lambda= d^{\frac{1}{6}}$. This is a slight spoiler but for experts in amplification it should be no surprise that this is the optimal size of the amplifier in this setting. Let $$S =\{ l \text{  prime }\colon l\asymp \Lambda \}.$$ By the prime number theorem (assuming $d$ is sufficiently large, which is no problem) we have $\sharp S \sim \Lambda/\log(\Lambda)$, but for us the following crude bound suffices 
\begin{equation}
	\sharp S \gg_{\epsilon} \Lambda^{1-\epsilon}. \nonumber
\end{equation}

Before we are ready to prove our key estimate we need to establish some counting results. Let
\begin{equation}
	M_z^{(\lambda)}(r) = \{ A\in \left[\begin{matrix} \Z &p^{\lambda}\Z \\ p^{\lambda} \Z & \Z	\end{matrix}\right]\colon \det(A)=r \text{ and } u(Az,z)\leq 1 \}.\nonumber
\end{equation}
The case $\lambda=0$ is easily handled using existing results. For example taking $N=\delta=1$ in \cite[Proposition~6.1]{templier_hybrid}. We follow standard procedure and write 
\begin{equation}
	M_z^{(\lambda)}(r)=M_{z,\star}^{(\lambda)}(r)\sqcup M_{z,p}^{(\lambda)}(r)\sqcup M_{z,u}^{(\lambda)}(r).\nonumber
\end{equation}
Here the subscript $\star$ indicated that we are dealing with \textit{generic} matrices $A=\left(\begin{matrix} a& b\\ c&d\end{matrix}\right)$ with $c\neq 0$ and $(a+d)^2\neq 4r$. On the other hand $u$ stands for \textit{unipotent} so that $A\in M_{z,u}^{(\lambda)}(r)$ if and only if $c=0$. Finally, $A\in  M_{z,p}^{(\lambda)}(r)$ if $c\neq 0$ and $(a+d)^2= 4r$. These are the \textit{parabolic} matrices. 

Counting the contribution of generic matrices is a standard lattice point counting argument, which we slightly modify. Note that our live is much easier, since we can take $z$ in the classical fundamental domain for $\SL_2(\Z)$.

\begin{lemmy} \label{generic}
For $\lambda>1$	and $y\geq \frac{\sqrt{3}}{2}$ we have
\begin{equation}
	\sum_{r\asymp K} [\sharp M_{z,\star}^{(\lambda)}(r)+M_{z,p}^{(\lambda)}(r)] \ll \frac{K^{\frac{3}{2}}}{p^\lambda}+\frac{K^2}{p^{2\lambda}}.\nonumber
\end{equation}
and 
\begin{equation}
	\sum_{\substack{r\asymp K, \\ r=\square}} \sharp M_{z,\star}^{(\lambda)}(r) \ll \frac{K^{1+\epsilon}}{p^\lambda}+\frac{K^{\frac{3}{2}+\epsilon}}{p^{2\lambda}}.\nonumber
\end{equation}
\end{lemmy}
We closely follow the argument in \cite{templier_hybrid}.
\begin{proof}
Write $A=\left(\begin{matrix} a&b \\ c & d\end{matrix}\right)$. Since $c\neq0$ and $c\ll \sqrt{r}y^{-1}$ we have $\ll \frac{K^{\frac{1}{2}}}{p^{\lambda}y}$ choices for $c$. Similarly, using the bound $\abs{a+d}\ll\sqrt{K}$, we have $\ll K^{\frac{1}{2}}$ possibilities to choose $a+d$. Finally, we have the bound
\begin{equation}
	\abs{-cz^2+(a-d)z+b}^2 \leq 4Ky^2.\nonumber
\end{equation}
Write $b=p^{\lambda}b'$ and consider the lattice $L=\langle z,p^{\lambda}\rangle$. Note that $L$ has cocolume $\asymp p^{\lambda}y$ and first successive minima $\gg 1$ (since $y\gg 1$). Let $B$ be the ball of radius $2\sqrt{K}y$ around $-cz^2$. Then we have
\begin{equation}
	\sharp\{ (b',a-d)\} \ll \sharp(B\cap L) \ll 1+\sqrt{K}y+\frac{Ky}{p^{\lambda}}.\nonumber
\end{equation}
We have counted the number of possibilities for the admissible quadruples $(c,b',a+d,a-d)$. Since each of those quadruples uniquely determines a matrix $A$ we have established
\begin{align}
	\sharp \{ A\in M_{z,\star}^{(\lambda)}(r)\colon r\asymp K \} &\ll \frac{K^{\frac{1}{2}}}{p^{\lambda}y}\cdot K^{\frac{1}{2}}\cdot  (1+\sqrt{K}y+\frac{Ky}{p^{\lambda}}) \nonumber \\
	&\ll  \frac{K}{p^{\lambda}}+\frac{K^{\frac{3}{2}}}{p^\lambda}+\frac{K^2}{p^{2\lambda}}.\nonumber
\end{align}

The case when we are only considering square matrices only needs a minor modification.  Indeed, instead of counting $a+d$ trivially as earlier we observe that
\begin{equation}
	(a-d)^2+4bc=(a+d)^2-4(\sqrt{r})^2. \nonumber
\end{equation}
If the matrix is parabolic the right hand side would be $0$. Thus we now consider only generic matrices. For those we can fix the left hand side first, so that we determine $(a+d,\sqrt{r})$ essentially as solutions to a generalised Pell equation. There are at most $\ll K^{\epsilon}$ possibilities.
\end{proof}

\begin{lemmy}\label{unipotent}
We have
\begin{equation}
	\sharp  M_{z,u}^{(\lambda)}(r) \ll r^{\epsilon}(1+\sqrt{r}p^{-\lambda}y). \nonumber
\end{equation}
\begin{align}
	&\sum_{l\in S}M_{z,u}^{(\lambda)}(l) \ll \Lambda+\frac{\Lambda^{\frac{3}{2}} y}{p^{\lambda}}, \nonumber \\
	&\sum_{l_1,l_2\in S}M_{z,u}^{(\lambda)}(l_1l_2) \ll \Lambda^2+\frac{\Lambda^3y}{p^{\lambda}} \text{ and }\nonumber\\
	&\sum_{l_1,l_2\in S}M_{z,u}^{(\lambda)}(l_1^2l_2^2) \ll \Lambda^2+\frac{\Lambda^{4}y}{p^{\lambda}}.\nonumber
\end{align}
\end{lemmy}
\begin{proof}
The first estimate follows analogously to \cite[(A.10)]{iwaniec-sarnak} using $p^\lambda\mid b$. Recall the bound $\abs{b}\ll \sqrt{r}y$ used in the process.

The other bounds are derived elementary using only the fact that $S$ contains only primes. (In contrast to \cite[Lemma~2.4]{harcos-templier_III} we do not need a lattice counting argument, because we have an additional congruence condition on $b$ that we can use.) We will only show the finial estimate, since the others are derived similarly.

There are $\ll \Lambda^2$ possible choices for $r=l_1^2l_2^2\asymp \Lambda^4$. Having fixed the determinant of this form we find that there are only $\ll 1$ choices for $(a,d)$ with $ad=l_1^2l_2^2$. Finally we observe that we can choose $b$ in $\ll 1+\frac{\Lambda^2y}{p^{\lambda}}$ ways, since $p^\lambda\mid b$ and $\abs{b}\ll \Lambda^2 y$. Putting these estimates together completes the proof.

\end{proof}

\begin{rem}\label{rem:vanish_c}
Suppose $y\geq \sqrt{3}{2}$. As in \cite[(A.7)]{iwaniec-sarnak} we have the bound 
\begin{equation}
	\abs{c}\leq \frac{\sqrt{8r}}{y} \leq \sqrt{\frac{2^5r}{3}}. \nonumber
\end{equation} 
Thus if $p^{\lambda}\geq 3,5\cdot\sqrt{r}$, we must have $c=0$. This is because $p^{\lambda}\mid c$. In this case we obtain
\begin{equation}
	M_z^{(\lambda)}(r)=M_{z,u}^{(\lambda)}(r). \nonumber
\end{equation}
\end{rem}

Finally we need to consider the parabolic contribution.

\begin{lemmy}\label{parabolic}
For $\lambda>1$	and $y\geq \frac{\sqrt{3}}{2}$ we have
\begin{equation}
	\sum_{\substack{r\asymp K, \\ r=\square}} \sharp M_{z,p}^{(\lambda)}(r) \ll \frac{K^{\frac{3}{2}}}{p^{\lambda}}.\nonumber
\end{equation}
\end{lemmy}
\begin{proof}
This follows along the lines of \cite[Lemma~14]{blomer-maga-harcos-milicevic_fields}.
\end{proof}

We can now prove the main estimate of this section.

\begin{lemmy}
Assume $p>3$. Suppose $\pi$ belongs to Case~1, ~2 or ~3.1. Then, for $\frac{\sqrt{3}}{2}\ll y\ll \sqrt{d}$ we have
\begin{equation}
	\Phi(x+iy) \ll \sqrt{T}d^{\frac{5}{6}}.\nonumber
\end{equation}
If $\pi$ belongs to Case~3.2, then we have the weaker bound
\begin{equation}
	\Phi(x+iy) \ll \sqrt{T}d^{\frac{11}{12}} \text{ for } \frac{\sqrt{3}}{2}\ll y\ll d^{\frac{1}{4}}.\nonumber
\end{equation}
\end{lemmy}
\begin{proof}
We start with Cases~1, ~2 or ~3.1. Our starting point is Corollary~\ref{pre-trace_inequality}. Breaking the $g$-sum up into pieces on which we can estimate the character suing Lemma~\ref{character} we get
\begin{equation}
	\frac{(\sharp S)^2}{\dim_{\C}\pi_p^{K_p(m_{\pi})}} \cdot \Phi(z)^2 \ll  T\sum_{0\leq \lambda \leq m_{\pi}}p^{\lambda} \sum_{r} \frac{\abs{y_r}}{\sqrt{r}}  \cdot \sharp M_z^{(\lambda)}(r).\nonumber
\end{equation}

We first consider the contribution of $\lambda=0$. In this case the counting problem is independent of $p$ and relatively easy. Indeed we have
\begin{equation}
	\sum_{r} \frac{y_r}{\sqrt{r}}\sharp \{A\in M_2(\Z) \colon \det(A)=r \text{ and }u(Az,z)\leq 1 \} \ll \Lambda^4+\Lambda^{\frac{5}{2}}y.
\end{equation}
This is for example \cite[Proposition~6.1]{templier_hybrid} with $N=\delta=1$ and $z\in\mathcal{F}$ so that $y\gg 1$.

Next we assume $\lambda>0$. We summarise the results from Lemma~\ref{generic}, \ref{parabolic} and \ref{unipotent} in Table~\ref{tab:table1} below.
\begin{table}[h!]
	\begin{center}
		\caption{Summary of counting results}
		\label{tab:table1}
		\begin{tabular}{l||c|c|c} 
			$\sum_{r\in T}\frac{\abs{y_r}}{\sqrt{r}}\cdot\sharp M_{z,\lozenge}^{(\lambda)}(r)$ & $\lozenge= u$ & $\lozenge = \star$ & $\lozenge = p$\\
			\hline \hline
			$T=\{1\}$ & $\Lambda+\frac{\Lambda y}{p^\lambda}$&0&0 \\
			\hline
			$T=\{ l\colon l\in S\}$ &$\Lambda^{\frac{1}{2}}+\frac{\Lambda y}{p^\lambda}$ & $\frac{\Lambda}{p^\lambda}+\frac{\Lambda^{\frac{3}{2}}}{p^{2\lambda}}$& included in $\star$ \\
			\hline
			$T=\{ l_1l_2\colon l_1,l_2\in S\}$ &$\Lambda+\frac{\Lambda^2 y}{p^\lambda}$ & $\frac{\Lambda^{2}}{p^\lambda}+\frac{\Lambda^{3}}{p^{2\lambda}}$ & included in $\star$ \\
			\hline
			$T=\{ l_1^2l_2^2\colon l_1,l_2\in S\}$ &$1+\frac{\Lambda^2 y}{p^\lambda}$ & $\frac{\Lambda^{2+\epsilon}}{p^\lambda}+\frac{\Lambda^{4+\epsilon}}{p^{2\lambda}}$  &  $\frac{\Lambda^4}{p^\lambda}$ \\
			\hline
			\hline
			\textbf{total} & $\Lambda+\frac{\Lambda^2 y}{p^\lambda}$ & $\frac{\Lambda^{2+\epsilon}}{p^\lambda}+\frac{\Lambda^{4+\epsilon}}{p^{2\lambda}}$ & $\frac{\Lambda^4}{p^\lambda}$\\
			\hline
		\end{tabular}
	\end{center}
\end{table}
Note that for the contribution of $r=1$ we have used Remark~\ref{rem:vanish_c}. 
All together this gives a contribution of
\begin{equation}
	\sum_{r}\frac{\abs{y}}{\sqrt{r}}\cdot \sharp M_z^{(\lambda)}(r) \ll \Lambda+\frac{\Lambda^4}{p^\lambda}+\frac{\Lambda^2y}{p^\lambda}.\nonumber
\end{equation}

Inserting these estimates in our (amplified) pre-trace inequality we get
\begin{equation}
	\frac{(\sharp S)^2}{\dim_{\C}\pi_p^{K_p(m_{\pi})}} \cdot \Phi(z)^2 \ll  T\left(\Lambda^4+\Lambda^{\frac{5}{2}}y+\sum_{1\leq \lambda \leq m_{\pi}}p^{\lambda}\left[\Lambda+\frac{\Lambda^4}{p^\lambda}+\frac{\Lambda^2y}{p^\lambda}\right]\right).\nonumber
\end{equation}
This simplifies to
\begin{equation}
	\Phi(z)^2 \ll  \frac{Td^{1+\epsilon}}{\Lambda^{2-\epsilon}}\left(\Lambda^{\frac{5}{2}}y+\Lambda^4+\Lambda d \right). \nonumber
\end{equation}
Inserting $\Lambda=d^{\frac{1}{3}}$ yields $$\Phi(z)^2 \ll Td^{1+\epsilon}(d^{\frac{2}{3}}+d^{\frac{1}{6}}y)$$ which directly implies the result for the non-exceptional cases ~1, ~2 and ~3.1.

Finally if $\pi$ belongs to Case~3.2, then the same analysis with the weaker character estimates yields
\begin{align}
	\frac{(\sharp S)^2}{\dim_{\C}\pi_p^{K_p(m_{\pi})}} \cdot \Phi(z)^2 &\ll  T\left(d^{\frac{1}{2}}\Lambda^4+d^{\frac{1}{2}}\Lambda^{\frac{5}{2}}y+\sum_{1\leq \lambda \leq m_{\pi}}p^{\frac{m_{\pi}+\lambda}{2}}\left[\Lambda+\frac{\Lambda^4}{p^\lambda}+\frac{\Lambda^2y}{p^\lambda}\right]\right) \nonumber \\
	&\ll T\left(d^{\frac{1}{2}}\Lambda^4+d^{\frac{1}{2}}\Lambda^{\frac{5}{2}}y + \Lambda d^{1+\epsilon}\right). \nonumber
\end{align}
This prompts the choice $\Lambda=d^{\frac{1}{6}}$ and we find $$\Phi(z)^2 \ll Td^{1+\epsilon}(d^{\frac{5}{6}}+d^{\frac{7}{12}}y)$$ This completes the proof.
\end{proof}

\textit{Proof of Theorem~\ref{th:main}:}
We are now ready to complete the proof of our main theorem. First of all note that it is enough to bound $\Phi(z)$ for $z\in\mathcal{F}$. Next, if $y\ll \sqrt{d}$ (resp. $y\ll d^{\frac{1}{4}}$), then we are done by the previous lemma. For larger $y$ the Whittaker expansion (see Lemma~\ref{lm:whitt_bound}) gives even better result. \qed


\bibliographystyle{amsplain}			
\bibliography{lb}				

\end{document}